\documentclass[12pt]{amsart}

\usepackage{amsfonts,amscd,amssymb,enumerate,enumitem,hyperref,cleveref}
\usepackage[normalem]{ulem}

\allowdisplaybreaks

\date{\today}

\newtheorem{thm}{Theorem}[section]
\newtheorem{cor}[thm]{Corollary}
\newtheorem{lem}[thm]{Lemma}

\newtheorem{prop}[thm]{Proposition}
\theoremstyle{definition}

\theoremstyle{remark}
\newtheorem{rem}[thm]{Remark}

\numberwithin{equation}{section}

\newcommand{\R}{\mathbb R}

\newcommand{\Na}{\mathbb N}
\newcommand{\He}{\mathbb H}

\newcommand{\C}{{\mathbb C}}

\newcommand{\s}{\mathcal{S}}

\newcommand{\Fs}{\mathcal{F}^\lambda(\C^{2n}) }

\renewcommand{\Re}{\operatorname{Re}}
\renewcommand{\Im}{\operatorname{Im}}
\newcommand{\tr}{\operatorname{tr}}

\usepackage{color}

\title[Quantum Hermite functions]
{Quantum Hermite functions \\
and Fourier transform of operators}

\author[R. Garg and S. Thangavelu]{Rahul Garg and Sundaram Thangavelu}

\address[R. Garg]{Department of Mathematics, Indian Institute of Science Education and Research Bhopal, Bhopal--462066, India.}
\email{rahulgarg@iiserb.ac.in}

\address[S. Thangavelu]{National Science Chair, Department of Mathematics, Indian Institute of Science, Bangalore--560012, India and Honorary Professor, University of Queensland, Brisbane, Australia.}
\email{veluma@iisc.ac.in}

\begin{document}

\begin{abstract}
We construct operator analogues of Hermite functions which form an orthonormal basis for the Hilbert space $ \mathcal{S}_2$ of Hilbert-Schmidt operators on $ L^2(\R^n).$ We use this orthonormal basis to define Fourier transform on $ \mathcal{S}_2 $ and study some of its basic properties.
\end{abstract}

\maketitle

\section{Introduction} \label{Sec-intro}
In this article we are concerned with operator analogues of Hermite functions forming an interesting orthonormal basis for the Hilbert space $ \s_2$  of all Hilbert-Schmidt operators on $ L^2(\R^n)$  equipped with  the inner product $ (S,T) = \tr(ST^\ast).$ It is well known that there is a one to one correspondence between $\s_2 $ and $ L^2(\R^{2n}) $ since every $ T \in \s_2 $ is an integral operator with kernel $ K \in L^2(\R^{2n}).$ A lesser known fact is that for every non-zero $ \lambda \in \R ,$ there is a unitary operator $ \pi_\lambda : L^2(\R^{2n}) \rightarrow \s_2 $ known as the Weyl transform. Consequently, every orthonormal basis for $ L^2(\R^{2n}) $ gives rise to an orthonormal basis for $ \s_2 $ and vice versa. For example, the complete orthonormal system of (scaled) Hermite functions $ \Phi_\mu^\lambda, \mu \in \mathbb N^{2n} $  leads to the  basis $ \pi_\lambda(\Phi_\mu^\lambda) .$ Another example is provided by the family of special Hermite functions  $ \Phi_{\alpha,\beta}^\lambda, \alpha, \beta \in \mathbb N^n.$  We do not have a good description of the operators $ \pi_\lambda(\Phi_\mu^\lambda) $ but the operators $ \pi_\lambda(\Phi_{\alpha,\beta}^\lambda) $ have  explicit descriptions in terms of the Hermite orthonormal basis  for $ L^2(\R^n).$ In this note we construct a new orthonormal basis $ \Psi_\mu^\lambda, \mu \in \mathbb N^{2n} $ for $ L^2(\R^{2n}) $ so that $ S_\mu^\lambda = \pi_\lambda(\Psi_\mu^\lambda) $ are very explicit having a definition analogous to that of Hermite functions, and hence we call them the operator analogues of Hermite functions or Quantum Hermite functions. Interestingly, the functions $ \Psi_\mu^\lambda $ turn out to be suitable dilations of the Hermite functions $ \Phi_\mu$ on $ \R^{2n}.$\\

We begin by recalling the definition of the normalised Hermite functions on $ \R^n.$ For any $ \alpha \in \mathbb N^n,$ let $ \partial^\alpha = \Pi_{j=1}^n \partial_j^{\alpha_j} $ where $ \partial_j = \frac{\partial}{\partial x_j} $ are the partial derivatives. We define the Hermite functions $ \Phi_\alpha $ and the Hermite polynomials $ H_\alpha$ by
\begin{equation} \label{hermite} 
\Phi_\alpha(x) = \left( \pi^{n/2} \alpha! 2^{|\alpha|}\right)^{-1/2} H_\alpha(x) \, e^{-\frac{1}{2}|x|^2} = (-1)^{|\alpha|} \left( \pi^{n/2} \alpha! 2^{|\alpha|}\right)^{-1/2} ( \partial^\alpha e^{-|x|^2})\, e^{\frac{1}{2}|x|^2} 
\end{equation} 
and for any $ \lambda \in \R, \lambda \neq 0,$ we let $ \Phi_\alpha^\lambda(x) = |\lambda|^{n/4} \Phi_\alpha(\sqrt{|\lambda|} x).$ These functions $\Phi_\alpha^\lambda $, $\alpha \in \mathbb N^n$, form an orthonormal basis for $ L^2(\R^n).$ The functions $ \Phi_\alpha^\lambda $ turn out to be eigenfunctions of the (scaled) Hermite operator $ H(\lambda) = -\Delta+\lambda^2 |x|^2$ with eigenvalues $ (2|\alpha|+n) |\lambda|.$ The Hermite operator has a self adjoint extension which generates a semigroup, called the Hermite semigroup, denoted by $ e^{-tH(\lambda)}, t >0.$ In defining the operator analogues of $ \Phi_\alpha^\lambda $ the role of the Gaussian $ e^{-|x|^2} $ will be played by  the operator $ e^{-H(\lambda)}.$ \\

On the space $ B(L^2(\R^n)) $ of all bounded linear operators on $ L^2(\R^n) $ we have a commuting family $ \mathcal{D}_j = \mathcal{D}_j(\lambda), j = 1, 2, \ldots, 2n $ of derivations. For $ \mu \in \mathbb N^{2n} $, we define $ \mathcal{D}^\mu = \Pi_{j=1}^{2n} \mathcal{D}_j^{\mu_j}$ and these operators play the role of the partial derivatives $ \partial^\alpha.$ 
In analogy with \eqref{hermite} we define  the operator analogues of the Hermite functions by 
 \begin{equation}\label{op-hermite}  S_\mu^\lambda =  \left( \pi^{n} \mu! 2^{|\mu|}\right)^{-1/2} ( \mathcal{D}^\mu e^{-H(\lambda)})\,  e^{\frac{1}{2}H(\lambda)}. 
\end{equation} 

\begin{thm} 
The family of operators $ S_\mu^\lambda, \mu \in \mathbb N^{2n}$ forms an orthonormal basis for $ \s_2.$
\end{thm}
 
The operators $ S_\mu^\lambda $ occur as the Weyl transforms of certain functions $ \Psi_\mu^\lambda \in L^2(\R^{2n})$ which  form an orthonormal basis for $ L^2(\R^{2n}).$ We will show that these functions are nothing but a suitable dilation of the Hermite functions on $ \R^{2n}.$ In fact
\begin{equation}
\Psi_\mu^\lambda(x,u) =  ( (\lambda/2) \coth (\lambda/2) )^{n/2} \, \Phi_\mu( \sqrt{ (\lambda/2) \coth (\lambda/2)}(x,u)).
\end{equation}
Actually we define the functions $ \Psi_\mu^\lambda $ first and $ S_\mu^\lambda $ will be defined as $ \pi_\lambda(\Psi_\mu^\lambda).$ We will show that $ S_\mu^\lambda $ satisfies equation \eqref{op-hermite}. By calculating the non-commutative derivations of $ e^{-H(\lambda)} $ it can be shown that 
\begin{equation}
S_\mu^\lambda  =  P_\mu^\lambda(A(\lambda), A^\ast(\lambda)) \, e^{-\frac{1}{2} H(\lambda)} 
\end{equation}
where $P_\mu^\lambda(A(\lambda), A^\ast(\lambda))$ is a non-commutative polynomial in the vectors $ A(\lambda)=(A_j(\lambda))$ and $ A^\ast(\lambda)=(A_j^\ast(\lambda)) $ of annihilation and creation operators.  The construction of $ \Psi_\mu^\lambda $ and hence of $ S_\mu^\lambda $ are motivated by a search for the twisted analogue of the following  fact which we recall now.\\ 
 
If we let $ \mathcal{F}(\C^n) $ stand for the Fock space consisting of all entire functions $ F $ that are square integrable with respect  to the Gaussian measure $ d\nu_n(z)= (4\pi)^{-n/2}\,e^{-\frac{1}{2}|z|^2} \, dz ,$ then the Bargmann transform $ B $ takes $ L^2(\R^n)$ unitarily onto $ \mathcal{F}(\C^n) .$  It is well known  that $ B $ takes the Hermite functions $ \Phi_\alpha $ onto the monomials $ \zeta_\alpha(z) = ( \pi^{n/2} 2^{|\alpha|} \alpha !)^{-1/2} \, z^\alpha.$  There is a family of (twisted) Fock spaces $ \mathcal{F}^\lambda(\C^{2n}) $ indexed by  $ \lambda \in \R $ which coincide with the classical Fock space $\mathcal{F}(\C^{2n})$ when  $ \lambda =0.$ These spaces are related  to the twisted Bergman spaces $ \mathcal{B}_t^\lambda(\C^{2n}) $ studied in \cite{KTX} which appear naturally in connection with the heat kernel transform or the Segal-Bargmann transform on the Heisenberg group $ \He^n.$  These spaces are defined in terms of  the weight function
$$ w_\lambda(z,w) = 4^n \, c_{n,\lambda} \, e^{\lambda \Im( z\cdot \bar{w})} \, e^{-\frac{1}{2} \lambda (\coth \lambda) |(z,w)|^2}, \quad (z,w) \in \C^{2n}.$$
Thus $ \mathcal{F}^\lambda(\C^{2n}) $ is  the space of all entire functions $ F $ on $\C^{2n} $ which are square integrable with respect to the measure $ w_\lambda(z,w) \, dz \, dw.$ 
We remark that the constant $ c_{n,\lambda} $ is chosen in such a way that  $ c_{n,0} = (4\pi)^{-n}$ but we take $ w_0(z,w) = (4\pi)^{-n}\, e^{-\frac{1}{2}  |(z,w)|^2}$  so that $ d\nu_{2n}(z,w) = w_0(z,w) dz\, dw.$ Note that $ \mathcal{F}^0(\C^{2n}) = \mathcal{F}(\C^{2n})$ is just the standard Fock space of entire functions on $ \C^{2n} .$ \\

The analogue of $ B $ is the twisted Bargmann transform $ B_\lambda ,$ studied in \cite{GT}, which takes $ L^2(\R^{2n}) $ unitarily onto $ \Fs.$   As discovered by the authors  of  \cite{BD}, there  exists  a very simple relation  between the spaces $\Fs $ and $\mathcal{F}(\C^{2n}).$  There is an invertible linear transformation $ \sigma_\lambda : \C^{2n} \rightarrow \C^{2n} $ such that $ F \circ \sigma_\lambda \in \Fs $ whenever $ F \in \mathcal{F}(\C^{2n}).$ In view of this $ \zeta_\mu \circ \sigma_\lambda \in \Fs $ and hence there exists $ \Psi_\mu^\lambda \in L^2(\R^{2n}) $ such that $ B_\lambda \Psi_\mu^\lambda = \zeta_\mu \circ \sigma_\lambda.$ Since $ \zeta_\mu, \mu \in \mathbb N^{2n} $ is an orthonormal basis for $ \mathcal{F}(\C^{2n}),$ it follows that $ \Psi_\mu^\lambda, \mu \in \mathbb N^{2n} $ forms an orthonormal basis for $ L^2(\R^{2n}) $ and hence $ S_\mu^\lambda = \pi_\lambda(\Psi_\mu^\lambda) $ an orthonormal basis for $ \s_2.$\\

One of the several interesting properties of the Fock space $ \mathcal{F}(\C^n) $ is the interplay  between the Fourier transform $ \mathcal{F}$ on $ \R^n $ and the Bargmann transform. More precisely, if $ U $ is the unitary operator defined on the Fock space by $ UF(z) = F(-iz),$ then $ U \circ B = B \circ \mathcal{F} .$ Motivated by this relation, in Section 3.5 we define  Fourier transform of  Hilbert-Schmidt operators. In order to do this, we make use of the Gauss-Bargmann transform $ \mathcal{G}_\lambda = B_\lambda \circ \pi_\lambda^\ast$ which takes $ \mathcal{S}_2 $ unitarily onto $ \Fs $. Once again the map $ U: \Fs \rightarrow \Fs $ defined by $ UF(z,w) = F(-iz,-iw) $ is unitary and hence  $ \mathcal{F} $ defined by the relation $ U \circ \mathcal{G}_\lambda = \mathcal{G}_\lambda \circ \mathcal{F} $ defines a unitary operator $ \mathcal{F} : \mathcal{S}_2 \rightarrow \mathcal{S}_2 $ which we call the Fourier transform of operators. We establish several basic properties of this Fourier transform. In particular, we show that the operators $ S_\mu^\lambda $ are eigenvectors for this Fourier transform, viz. $ \mathcal{F}(S_\mu^\lambda) = (-i)^{|\mu|}  S_\mu^\lambda.$ This is the analogue of the well known property $ \mathcal{F}(\Phi_\alpha) =(-i)^{|\alpha|} \Phi_\alpha$  satisfied by the Hermite functions.\\

The plan of the paper is as follows. In Section 2 we  set up the notation and recall  the preliminary material, especially  results on twisted Fock spaces and their relation to the standard Fock spaces. In Section 3 we introduce the operator analogues of the Hermite functions and establish their basic properties. In subsection 3.5 we define Fourier transform of operators and bring out the analogy between classical Fourier transform and the new Fourier transform on $ \mathcal{S}_2.$ As an example we state and prove an exact analogue of Hardy's theorem for the Fourier transform of operators. We also describe Fourier transform of radial operators and list some important problems to be studied in a forthcoming paper.

\section{ Preliminaries} \label{Sec-prelim}

\subsection{Hermite functions, Bargmann transform and the Fock space} 
For all the properties of the Bargmann transform we  use in this article, we refer to the paper \cite{B} of Bargmann. For the results on Hermite functions, we refer to Szego \cite{GS} and the monograph \cite{ST-princeton}.

\medskip 
Consider  the monomials  $ \zeta_\alpha(w) =  ( \pi^{n/2} 2^{|\alpha|} \alpha !)^{-1/2} w^\alpha,\, w \in \C^n .$  They form an orthonormal system in $ L^2(\C^n, d\nu_n) $ where  $ d\nu_n(w) = (4\pi)^{-n/2}\, e^{-\frac{1}{2}|w|^2} \, dw.$ Then the Hermite functions $ \Phi_\alpha $ defined in \eqref{hermite} satisfy
 the generating function identity 
 \begin{equation}\label{gen-id} 
   \sum_{\alpha \in \Na^n} \Phi_\alpha(x)\, \zeta_\alpha(w)  = \pi^{-n/2} e^{\frac{1}{4}w^2} e^{-\frac{1}{2}(x-w)^2}.
 \end{equation}
This can be easily verified in the one dimensional case by Taylor expanding  the   entire function $  g_x(w) = e^{-w^2+2xw} $ and noting that
 $$ e^{-w^2+2xw} = \sum_{k=0}^\infty \frac{1}{k!} H_k(x)\, w^k .$$
Since $ g_x(w) = e^{x^2} e^{-(x-w)^2} $, it is clear that $ \frac{d^k}{dw^k}g_x(0) = H_k(x)$ is a polynomial of degree $ k.$   The higher dimensional case follows from the one dimensional case as $ H_\alpha(x) = \Pi_{j=1}^nH_{\alpha_j}(x_j).$
 It is well  known that this family $ \Phi_\alpha, \alpha \in \Na^n $ forms an orthonormal basis for $ L^2(\R^n).$   We can  use the above identity to define the Bargmann transform and deduce  many of its properties easily. \\
 
  For $ f \in L^2(\R^n) $ we define its Bargmann transform $ Bf $  by integrating $f $ against the kernel on the right hand side of the generating function above. Thus
 $$ Bf(w) = \pi^{-n/2} e^{\frac{1}{4} w^2} \int_{\R^n} f(x) \,e^{-\frac{1}{2}(x-w)^2} dx =  \pi^{-n/2} e^{-\frac{1}{4} w^2} \int_{\R^n} f(x) \,e^{-\frac{1}{2}x^2}\, e^{x\cdot w} dx.$$
 In view of the above identity, it is clear that $ B\Phi_\alpha = \zeta_\alpha $  and we have the expansion 
 $$ Bf(w) =  \sum_{\alpha \in \Na^n} (f, \Phi_\alpha)\, \zeta_\alpha(w). $$ 
 As the functions $ \zeta_\alpha $ form an orthonormal system in $ L^2(\C^n, d\nu),$ from the Plancherel theorem for the Hermite expansions we deduce that
 $$ \int_{\C^n} |Bf(w)|^2\, d\nu_n(w) =  \sum_{\alpha \in \Na^n} |(f, \Phi_\alpha)|^2\, = \int_{\R^n} |f(x)|^2\, dx.$$
 Thus for $ f \in L^2(\R^n) $ its Bargmann transform $ Bf $ is an entire function which is square integrable with respect to $ d\nu_n.$ \\
 
We define the Fock space $ \mathcal{F}(\C^n) $ as the subspace of $ L^2(\C^n, d\nu) $ consisting of entire functions. Then it is clear that $ B: L^2(\R^n) \rightarrow \mathcal{F}(\C^n) $ is an isometry.  It can be shown that any $ F \in \mathcal{F}(\C^n) $ is of the form $ Bf $ for a unique $ f \in L^2(\R^n).$ This makes $ B $ unitary  whose inverse is given by its adjoint $ B^\ast.$  The reproducing kernel for $ \mathcal{F}(\C^n) $ is given by $ \pi^{-n/2} e^{\frac{1}{2}z \cdot \bar{w}}.$ This means that for any $ F \in \mathcal{F}(\C^n) $ we have
$$ \pi^{-n/2}\, \int_{\C^n} F(w)\, e^{\frac{1}{2}z \cdot \bar{w}}\, d\nu_n(w) \, dw = F(z).$$

\subsection{Heisenberg group and the Weyl transform} We begin by recalling some well known facts about the Heisenberg groups $ \He^n $ and their representations. For more details we refer to  the monographs \cite{GBF, ST-uncertainty}. The group $  \He^n$ is  just $ \R^{2n} \times \R  =  \C^n \times \R $ equipped with the group law
$$ (z,t)(w,s) =  (z+w, t+s+\frac{1}{2} \Im(z \cdot \bar{w})).$$
Note that  if we let $ z = x+iy, w = u+iv,$ then $ \Im(z \cdot \bar{w}) =(u \cdot y-x \cdot v)$ is nothing but the symplectic form $ [\xi,\eta] $ on $ \R^{2n}$ between $ \xi = (x,u)$ and $ \eta =(y,v).$ Therefore, the group law can also be written in the form
$$ (\xi,t)(\eta,s) =(\xi+\eta, t+s+\frac{1}{2}[\xi,\eta]).$$
The group $ \He^n $ is a step two nilpotent Lie group, it is therefore unimodular and the Haar measure is simply given by the Lebesgue measure on $ \R^{2n+1}.$  The  spaces $ L^p(\He^n) $ are defined with respect to the Lebesgue measure. The convolution between two functions $ f, g \in L^1(\He^n) $ is defined in the usual way
$$ f \ast g(\xi,t) = \int_{\He^n} f ( (\xi,t)(\eta,s)^{-1}) \, g(\eta,s)\, d\eta\, ds.$$

The convolution on the Heisenberg group gives rise to a family of convolutions for functions on $ \R^{2n} $ known as twisted convolutions. To define them let us set up some notation.
For a function $ f \in L^1(\He^n) $ let us denote by $ f^\lambda$ the inverse Fourier transform (upto a constant multiple) of $ f $ in the central variable given by 
$$ f^\lambda(\xi) = \int_{-\infty}^\infty  f(\xi,t)\, e^{i\lambda t}\, dt.$$
Since $ (\eta,s)^{-1} = (-\eta,-s) $ a simple calculation shows that
$$ (f \ast g)^\lambda(\xi) = \int_{\R^{2n}} f^\lambda(\xi-\eta)\, g^\lambda(\eta)\,e^{i\frac{\lambda}{2} [\xi,\eta]}\, d\eta .$$
The right hand side defines what is known as the $ \lambda$-twisted convolution of $ f^\lambda $ with $ g^\lambda $, and is denoted by $ f^\lambda \ast_\lambda g^\lambda.$
By defining  the $\lambda$-twisted translation $ \tau_\lambda(\eta)$ by the equation 
\begin{align} \label{def:twisted-translation}
\tau_\lambda(\eta)g^\lambda(\xi) = g^\lambda(\xi-\eta)\, e^{-i\frac{\lambda}{2} [\xi,\eta]}, 
\end{align}
we see that  the twisted convolution $ f^\lambda \ast_\lambda g^\lambda $  is given by the following integral
\begin{equation}\label{t-con} f^\lambda \ast_\lambda g^\lambda(\xi) = \int_{\R^{2n}} f^\lambda(\eta) \,  \tau_\lambda(\eta) g^\lambda(\xi) \, d\eta. \end{equation}
Taking partial Fourier transform in the central variable is a useful technique which reduces problems on  $ \He^n$ to problems on $ \R^{2n}.$

\medskip 
We now briefly recall the representation theory of $ \He^n$ which is needed in order to define the group Fourier transform. For each $ \lambda \in \R, \lambda \neq 0,$ there is an irreducible unitary representation $ \pi_\lambda $ of $ \He^n$ on $ L^2(\R^n) $ explicitly given by 
$$ \pi_\lambda(x,y,t)f(\xi) = e^{i\lambda t} e^{i\lambda(x \cdot \xi +\frac{1}{2} x\cdot y)} f(\xi+y), \quad f \in L^2(\R^n).$$
We use the notation $ \pi_\lambda(x,y) $ for the operator $ \pi_\lambda(x,y,0).$  The Fourier transform of a function $ f \in L^1(\He^n) $ is the operator valued function 
$$ \hat{f}(\lambda) =: \int_{\He^n} f(\xi,t) \, \pi_\lambda(\xi,t)\,\, d\xi\, dt.$$
Recalling the definition of $ \pi_\lambda $ and $ f^\lambda $ we see that $ \hat{f}(\lambda) = \pi_\lambda(f^\lambda) $ where 
$$ \pi_\lambda(f^\lambda) =: \int_{\R^{2n}} \, f^\lambda(\xi)\, \pi_\lambda(\xi)\, d\xi $$
is known as the Weyl transform of $ f^\lambda.$ Properties of the Fourier transform $ f \rightarrow \hat{f} $ on $ \He^n$ are proved by studying the Weyl transform. \\

The Weyl transform initially defined on $ L^1(\R^{2n}) \cap  L^2(\R^{2n}) $  has an extension to the whole of $ L^2(\R^{2n})$ as a Hilbert-Schmidt operator.
It takes $ L^2(\R^{2n}) $ onto the the space $ \mathcal{S}_2 $ of Hilbert -Schmidt operators acting on $ L^2(\R^n).$ The inversion formula for the Weyl transform reads as
\begin{equation}\label{weyl-inversion}
f(x,u) = (2\pi)^{-n/2}\, |\lambda|^{n/2}\, \tr (\pi_\lambda(-x,-u) \pi_\lambda(f)).
\end{equation}
It can be shown that  $ \pi_\lambda(f \ast_\lambda g) = \pi_\lambda(f)\, \pi_\lambda(g) .$ 
Let us consider  the position and momentum operators of quantum mechanics,
$$ Q_j \varphi(\xi)= \xi_j\, \varphi(\xi) \quad \text{and} \quad P_j\varphi(\xi) = \partial_{\xi_j}\varphi(\xi)$$ 
acting on suitable functions $\varphi$ on $\R^n$. Also, for any two operators $A$ and $B$, let $ [A,B] $ stand for the commutator $ AB-BA.$ The relations stated in the next proposition play an important role in this article.

\begin{prop} For any $ g \in L^2(\R^{2n}) $ for which $ x_jg, u_j g \in L^2(\R^{2n}) $ we have 
$$  \pi_\lambda( \lambda x_jg) = i [ \pi_\lambda(g), P_j],\,\,\,\, \pi_\lambda(  u_jg) = [ \pi_\lambda(g), Q_j]. $$
\end{prop}
\vskip0.1in
For a proof of the above relations, see \cite[Proposition 2.3]{LLT}. The operators $ [ \pi_\lambda(f), P_j] $ and $ [ \pi_\lambda(f), Q_j] $ have been used to study Weyl multipliers and proving Paley-Wiener theorems 
on the Heisenberg group, see the articles \cite{GM, ST-JFA, BT}. We also require  some more properties of the Fourier transform on the Heisenberg group which we state below. Let 
$$ \widetilde{X}_j = \frac{\partial}{\partial x_j}-\frac{1}{2}y_j \frac{\partial}{\partial t}, \quad \widetilde{Y}_j = \frac{\partial}{\partial y_j}+\frac{1}{2}x_j \frac{\partial}{\partial t}, \quad T = \frac{\partial}{\partial t} $$
be the right invariant vector fields on the Heisenberg group forming a basis for its Lie algebra. Then the following relations have been proved in \cite[Proposition 2.3]{LLT}. 

\begin{prop} 
Under suitable conditions on functions $f$ on $\He^n$, we have 
$$ \widehat{(i \widetilde{X}_j f)}(\lambda) = \lambda \, Q_j \, \widehat{f}(\lambda) \quad \text{and} \quad \widehat{(-\widetilde{Y}_j f)}(\lambda) = P_j \, \widehat{f}(\lambda). $$
\end{prop}
\medskip
Recalling that   $\widehat{f}(\lambda) =  \pi_\lambda(f^\lambda) $ for any  function $ f $ on $ \He^n $ and defining 
$$ \widetilde{X}_j^\lambda = \frac{\partial}{\partial x_j}+ \frac{i}{2} \lambda y_j ,\,\,\,  \widetilde{Y}_j^\lambda = \frac{\partial}{\partial y_j}-\frac{i}{2} \lambda x_j  $$ 
we can easily obtain the following relations:
$$  \pi_\lambda(i \widetilde{X}_j^\lambda f^\lambda) = \lambda\, Q_j \pi_\lambda(f^\lambda) \quad \text{and} \quad \pi_\lambda(-\widetilde{Y}_j^\lambda f^\lambda) = P_j\, \pi_\lambda(f^\lambda) $$ 
By defining  the complex vector  fields 
$$ Z_j^\lambda = (\widetilde{X}_j^\lambda +i\widetilde{Y}_j^\lambda ) \quad \text{and} \quad  \overline{Z}_j^\lambda =(\widetilde{X}_j^\lambda -i\widetilde{Y}_j^\lambda ) $$ 
the  above formulas  can also be written in the form
\begin{equation} 
\label{Z-j-action}
\pi_\lambda( Z_j^\lambda f^\lambda) = -i  A_j(\lambda) \pi_\lambda(f^\lambda) \quad \text{and} \quad \pi_\lambda( \overline{Z}_j^\lambda f^\lambda) =  -i A_j^\ast(\lambda) \pi_\lambda(f^\lambda)
\end{equation}
where $A_j(\lambda)$ and $A_j^\ast(\lambda)$ are the annihilation and creation operators of quantum mechanics:
\begin{equation}\label{anni-crea}  A_j(\lambda) = \partial_{\xi_j}  +|\lambda| \xi_j \quad \text{and} \quad A_j^\ast(\lambda) = -\partial_{\xi_j}  +|\lambda| \xi_j. \end{equation}

\subsection{The twisted Bargmann transform and the twisted Fock spaces}
It is well known that the Bargmann transform $ B $ is closely related to the Schr\"odinger representation $ \pi $ of  $ \He^n$  which is unitarily equivalent to a representation realised on the Fock space $ \mathcal{F}(\C^n).$  In a similar way, for each non-zero $ \lambda \in \R $ we have a nilpotent Lie group $ \He^n_\lambda,$  called the twisted Heisenberg group and a unitary representation $ \rho_\lambda $ realised on a space $ \Fs ,$  known as the twisted Fock space and a twisted version of the Bargmann transform denoted by $ B_\lambda $ which takes $ L^2(\R^{2n}) $ unitarily onto $ \Fs.$  We begin by recalling the necessary background material to state relevant results on $ \Fs.$ The basic reference for this section is \cite{KTX, GT, ST-IJPAM}. We also refer the reader to \cite{ST-princeton, ST-uncertainty} for results related to Hermite and special Hermite operators.\\

Along with the Hermite operator $ H(\lambda) $ on $ \R^n$  we also consider the special Hermite operator $ L_\lambda $ on $ \R^{2n} $ generating the semigroup $ e^{-t L_\lambda}.$  It is known that the (twisted) convolution kernel associated with this semigroup is given by
$$ p_t^\lambda(x,u) = (4\pi)^{-n} \, \lambda^n (\sinh t\lambda)^{-n} \, e^{-\frac{1}{4}\lambda (\coth t\lambda)(|x|^2+|u|^2)},\,\,\,  $$
so that $ e^{-tL_\lambda}f = f \ast_\lambda p_t^\lambda.$ The kernel $ p_t^\lambda(x,u)$  has  a holomorphic extension $ p_t^\lambda(z,w) $ to $ \C^{2n}.$   The twisted Bargmann transform $ B_\lambda $ is defined on $ L^2(\R^{2n}) $ by
\begin{align} \label{def:Gauss-Bargmann-transform-tiwsted-fock}
B_\lambda f(z,w) = (2\pi)^{-n/2}\, |\lambda|^{n/2}\,  c_{n,\lambda} \,\, p_1^\lambda(z,w)^{-1}\, \tr\left( \pi_\lambda(-z,-w) \pi_\lambda(f) e^{-\frac{1}{2}H(\lambda)} \right) 
\end{align}
where $ c_{n,\lambda} = (4\pi)^{-n} \, \lambda^n (\sinh \lambda)^{-n}.$  
In view of the inversion formula \eqref{weyl-inversion} for the Weyl transform, using the fact that $ \pi_\lambda(p_t^\lambda) = e^{-tH(\lambda)} ,$  the twisted  Bargmann transform is also given by
$$ B_\lambda f(z,w) =  \,c_{n,\lambda} \,\, p_1^\lambda(z,w)^{-1}\,  f \ast_\lambda p_{1/2}^\lambda(z,w). $$
The image of $ L^2(\R^{2n}) $ under $ B_\lambda $ is known to be a weighted Bergman space with the explicit weight function
$$ w_\lambda(\zeta) = 4^n c_{n,\lambda}\, \, e^{-\frac{1}{2} \lambda (\coth \lambda)|\zeta|^2}\, e^{\lambda [ \Re \zeta,\,\Im \zeta ]},\,\,  \zeta = (z,w). $$
Here $ [\xi,\eta] $ is the symplectic form on $ \R^{2n} $ defined by $ [(x,u),(y,v)] = (u\cdot y-v\cdot x)$ and the constant $c_{n,\lambda} $ has been chosen in such a way that $ B_\lambda $ is an isometry: 
$$ \int_{\C^{2n}} |B_\lambda f(z,w)|^2\, w_\lambda(z,w) \, dz\, dw = \int_{\R^{2n}} |f(x,u)|^2\, dx\, du.$$ 

The twisted Fock space $ \Fs $ is defined as the space of all entire functions on $ \C^{2n} $ square integrable with respect to the measure $ w_\lambda(\zeta)\, d\zeta $ where $ d\zeta $ is the Lebesgue measure on $ \C^{2n}.$  When equipped with the inner product 
$$ \langle F, G \rangle = \int_{\C^{2n}} F(\zeta) \, \overline{G(\zeta)} \, w_\lambda(\zeta) \, d\zeta ,$$ 
the space $ \Fs $ becomes a reproducing kernel Hilbert space and the Bargmann transform
$ B_\lambda: L^2(\R^{2n}) \rightarrow \Fs $ becomes  a unitary operator. The reproducing kernel for $ \Fs$ is known explicitly.  For $ \zeta = (z,w), \zeta^\prime = (z^\prime,w^\prime),$ we define $ [\zeta, \zeta^\prime] =
(w \cdot z^\prime - z \cdot w^\prime)  $ be the extension of the symplectic for $\R^{2n}$ to $ \C^{2n}.$ With this notation, the reproducing kernel is  given by
\begin{equation}\label{rep-ker}
K_\zeta^\lambda(\zeta^\prime) = d_{n,\lambda}\, \,e^{ \frac{\lambda}{2} (\coth \lambda)\, (\overline{\zeta} \cdot \zeta^\prime)}\, e^{i\frac{\lambda}{2}\,  [ \overline{\zeta}, \zeta^\prime]} 
\end{equation}
 where $ d_{n,\lambda} $ is an explicit constant depending  on $ \lambda $ and the dimension. We remark that  when $ \lambda =0$ the kernel $ K_\zeta^0(\zeta^\prime)$ is a constant multiple of  the reproducing kernel for the Fock space $ \mathcal{F}(\C^{2n}).$ We refer the reader to \cite{KTX, GT} for more details of this transform and further properties of $ \Fs.$ We note that when $ \lambda = 0$  the twisted Fock space reduces to the classical Fock space $ \mathcal{F}(\C^{2n}).$\\

\begin{rem} Since $ p_{1/2}^ \lambda \ast_\lambda p_{1/2}^\lambda = p_1^\lambda $ it follows that $ B_\lambda p_{1/2}^\lambda = c_{n,\lambda}$ and hence the constant  $ c_{n,\lambda} $ is determined by 
the equation
$$ 4^n  c_{n,\lambda}^3\,  \int_{\C^{2n}} e^{-\frac{1}{2} \lambda (\coth \lambda)|\zeta|^2}\, e^{\lambda [ \Re \zeta,\,\Im \zeta ]}\, d\zeta =  \int_{\R^{2n}} |p_{1/2}^\lambda(x,u)|^2 \, dx\, du.$$
In particular, in the limiting case we get $ c_{n,0} = (4\pi)^{-n} $ so that $ w_0(z,w) dz dw = d\nu_0(z,w).$ Similary, the constant $ d_{n,\lambda} $ in the definition of $ K_\zeta^\lambda(\zeta^\prime)$ is given by the equation
$$ F(\zeta) = 4^n d_{n,\lambda} \int_{\C^{2n}} F(\zeta^\prime)\,\,e^{ \frac{\lambda}{2} (\coth \lambda)\, (\overline{\zeta^\prime} \cdot \zeta)}\, e^{i\frac{\lambda}{2}\,  [ \overline{\zeta^\prime}, \zeta]} \, w_\lambda(\zeta^\prime)\, d\zeta^\prime.$$ 
By taking $ F(\zeta) = c_{n,\lambda} $ we can easily evaluate the constant $ d_{n,\lambda}.$ By passing to the limit as $ \lambda $ goes to zero we see that $ 4^n K_\zeta^0(\zeta^\prime) $ is the reproducing kernel for $ \mathcal{F}(\C^{2n}) $ and hence $ d_{n,0} = (4\pi)^{-n}.$
\end{rem}

\subsection{ Derivations on the space of bounded linear operators}We denote by $ B(L^2(\R^n))$ the space of all bounded linear operators on $ L^2(\R^n).$ Given a densely defined  linear operator  $ A $  and 
 $ T \in  B(L^2(\R^n)) $ we define $ \partial_A(T) = [T,A]= TA-AT $ whenever it makes sense as an element of $B(L^2(\R^n)).$  Then it is easily checked that $ \partial_A(ST) = \partial_A(S) T + S \partial_A(T)$ and for this reason $ \partial_A $ is called a derivation.
 Two derivations $ \partial_A $ and $ \partial_B $ do not commute in general. Since
$ [ [T,A], B] = TAB- ATB-BTA+BAT ,$  we can rewrite the above  to get 
$$ \partial_B \circ \partial_A (T) =  T [A,B]+ TBA-ATB-BTA-[A,B]T+ABT.$$
From the above relation, we easily deduce that  $ \partial_A $ commutes with $ \partial_B  $ whenever $ [A,B] = cI.$\\

The annihilation and creation operators defined in \eqref{anni-crea} earn their name  because of their action on the Hermite functions as shown below:
$$ A_j(\lambda)\Phi_\alpha^\lambda =  \sqrt{2\alpha_j |\lambda|}\, \Phi_{\alpha-e_j}^\lambda \quad \text{and} \quad A_j(\lambda)^\ast\Phi_\alpha^\lambda =  \sqrt{(2\alpha_j+2)|\lambda|}\, \Phi_{\alpha+e_j}^\lambda. $$ 
The derivations $ \partial_{A_j(\lambda)} $ and $ \partial_{A_j^\ast(\lambda)}$ play an important role in this work.  As can be easily seen $ [A_j(\lambda), A_k^\ast(\lambda)] =  2 |\lambda| \delta_{jk} \, I$, and hence it follows that  $ \{ \partial_{A_j(\lambda)} ,  \partial_{A_j^\ast(\lambda)}: j = 1, 2, \ldots, n \} $ is a commuting family of derivations. The action of these derivations on the Hermite semigroup are easy to calculate.

\begin{lem}\label{lem.2.4} For any $ \lambda \neq 0, t >0 $ and $ j=1, 2, \ldots, n$ we have
$$ \partial_{A_j(\lambda)} e^{-tH(\lambda)} = (e^{2t|\lambda|}-1) A_j(\lambda) e^{-tH(\lambda)},\,\, \partial_{A_j(\lambda)^\ast} e^{-tH(\lambda)} = (e^{-2t|\lambda|}-1) A_j(\lambda)^\ast e^{-tH(\lambda)}.$$
\end{lem}
\begin{proof} The proof follows by direct calculation using the action of the creation and annihilation operators on the Hermite functions given in the previous paragraph.
\end{proof}


\section{Operator analogues of Hermite functions} 

\subsection{Deforming the Fock space} In this subsection we show that there exits a very simple relation between the Fock space $ \mathcal{F}(\C^{2n})$ and the twisted Fock space $ \Fs$ which plays a major role in defining the operator analogues of Hermite functions. We are indebted to Venku Naidu (personal communication) for bringing this to our attention.

\medskip 
For nonzero $ \lambda \in \R $, let $ \delta_\lambda $ stand for the dilation $ \delta_\lambda \zeta = (\frac{\lambda}{\sinh \lambda})^{1/2} \zeta$ for $ \zeta \in \C^{2n}.$ We also define $ c_\lambda = \cosh \lambda $ and $ s_\lambda = \sinh \lambda.$
Let $ \sigma_\lambda : \C^{2n} \rightarrow \C^{2n} $ stand for  the linear transformation  defined by 
$$  \sigma_\lambda(z,w) =  \delta_\lambda(c_{\lambda/2}z-i s_{\lambda/2}w, c_{\lambda/2} w+i s_{\lambda/2} z).$$ 
Then we have the following result which shows that $\Fs $ can be considered as a deformation of $ \mathcal{F}(\C^{2n}).$

\begin{prop} 
The map $ T_\lambda $ defined on $ \mathcal{F}(\C^{2n}) $ by 
$ T_\lambda F(\zeta) = a_{n,\lambda}\, F( \sigma_\lambda \zeta)$ where $ a_{n,\lambda} = \pi^{n/2} \sqrt{c_{n,\lambda}} $, is an isometric isomorphism  from $ \mathcal{F}(\C^{2n}) $ onto $ \Fs.$
\end{prop}
\begin{proof} 
For any $ F \in \mathcal{F}(\C^{2n}) $, consider the integral 
$$ \int_{\C^{2n}} |F(\zeta)|^2\, w_0(\zeta) \, d\zeta.$$ 
Since the determinant of the $4n \times 4n$ real matrix associated to the linear transformation  
$$ (z,w) \rightarrow (c_{\lambda/2}z-i s_{\lambda/2}w, c_{\lambda/2} w+i s_{\lambda/2} z)$$ 
is unity, we can make a change of variables in the last integral to get
$$ \int_{\C^{2n}} |F(\zeta)|^2\, w_0(\zeta) \, d\zeta = \left( \frac{\lambda}{\sinh \lambda} \right)^{2n}\, \int_{\C^{2n}} |F(\sigma_\lambda \zeta)|^2 \, w_0(\sigma_\lambda \zeta) \, d\zeta. $$ 
In view of the easily verifiable relation $ c_{n,0}^{-1}\, w_0(\sigma_\lambda \zeta) =  4^{-n}c_{n,\lambda}^{-1}\, w_\lambda(\zeta),$ the above identity implies that whenever $ F \in \mathcal{F}(\C^{2n}),$ we have 
$$ \int_{\C^{2n}} |T_\lambda F(\zeta)|^2\, w_\lambda(\zeta) \, d\zeta = \int_{\C^{2n}} |F(\zeta)|^2\, w_0(\zeta) \, d\zeta $$
As the map is clearly invertible, the proposition stands proved.
\end{proof}

\begin{rem} \label{reproduce} 
By direct calculation we can also verify that
$$ e^{\frac{1}{2} (\sigma_\lambda \zeta^\prime)\cdot \overline{(\sigma_\lambda \zeta)}} =  e^{\frac{1}{2}\lambda (\coth \lambda) ( \zeta^\prime \cdot \bar{\zeta}) } e^{-i \frac{\lambda}{2} [ \zeta^\prime, \bar{\zeta}]} .$$
Therefore, for any $ F \in \mathcal{F}(\C^{2n}),$ using the fact that $ d_{n,0} \, e^{\frac{1}{2} \zeta^\prime \cdot \bar{\zeta}}$ is the reproducing kernel, we have
$$ F(\sigma_\lambda \zeta) = (4\pi)^{-n} \, \int_{\C^{2n}} F(\zeta^\prime) \, e^{\frac{1}{2} (\sigma_\lambda \zeta) \cdot \bar{\zeta^\prime}} \, w_0(\zeta^\prime) d\zeta^\prime.$$
As before we make the same change of variables to rewrite the above as
$$ F(\sigma_\lambda \zeta) = (4\pi)^{-n}\, \Big(\frac{\lambda}{\sinh \lambda}\Big)^{2n} \, \int_{\C^{2n}} F(\sigma_\lambda \zeta^\prime) \,e^{\frac{1}{2} \sigma_\lambda\zeta \cdot \overline{\sigma_\lambda \zeta^\prime}}\, w_0(\sigma_\lambda \zeta^\prime) \, d \zeta^\prime.$$
In view of the above calculation this gives the reproducing formula
$$ T_\lambda F(\zeta) =  \int_{\C^{2n}} T_\lambda F(\zeta^\prime) \, \overline{K_\zeta^\lambda(\zeta^\prime)} \, w_\lambda(\zeta^\prime) \, d\zeta^\prime .$$
\end{rem}

\begin{rem} The relations described above between the Fock and twisted Fock spaces have some interesting consequences. For example,
consider the following Toeplitz operator $ T_g $ defined on $ \Fs$:
$$ T_gF(\zeta) = \int_{\C^{2n}}  g(\zeta^\prime)\, F(\zeta^\prime)\, e^{\frac{1}{2}\lambda \coth \lambda (  \zeta \cdot \bar{\zeta^\prime}) } e^{-i \frac{\lambda}{2} [\zeta, \bar{\zeta^\prime}]} \,w_\lambda(\zeta^\prime)\, d\zeta^\prime.$$
If we let $ F_0(\zeta) = F(\sigma_\lambda^{-1}\zeta) $ and $ g_0(\zeta) = g(\sigma_\lambda^{-1} \zeta),$ then it can be checked that
$$ T_{g_0}F_0(\sigma_\lambda \zeta) = \int_{\C^{2n}} g_0(\zeta^\prime)\, F_0(\zeta^\prime) \, e^{\frac{1}{2} (\sigma_\lambda \zeta) \cdot \bar{\zeta^\prime}} \, w_0(\zeta^\prime) \, d\zeta^\prime  =C_\lambda\,  T_gF(\zeta).$$
Thus the Toeplitz operator $ T_g $ is bounded on $ \Fs $ if and only if $ T_{g_0} $ is bounded on $ \mathcal{F}(\C^{2n})$, and the study of Toeplitz operators on $ \Fs $ can be reduced to the study of such operators on $ \mathcal{F}(\C^{2n}).$
\end{rem}

\subsection{Operator analogues of Hermite functions}  For every multi-index $ \mu \in \mathbb N^{2n} $ let 
$$   \zeta_\mu(z,w) =  c_\mu\, (z,w)^\mu, \, \,\, c_\mu = \pi^{-n/2} (2^{|\mu|} \mu!)^{-1/2}.$$ Then it is known that $ \zeta_\mu, \mu \in \mathbb N^{2n} $ forms an orthonormal basis for $ \mathcal{F}(\C^{2n}).$ Therefore, $ T_\lambda \zeta_\mu, \, \mu \in \mathbb N^{2n} $ forms an orthonormal system in $ \Fs.$ Consequently, there exist functions $ \Psi_\mu^\lambda \in L^2(\R^{2n})$ such that $ B_\lambda \Psi_\mu^\lambda(z,w) = a_{n,\lambda}\, \zeta_\mu(\sigma_\lambda(z,w)).$  Recalling the definition of $ a_{n,\lambda} $ this translates into the equation
$$ \zeta_\mu (\sigma_\lambda(z,w)) = \pi^{-n/2} \sqrt{c_{n,\lambda}}\, \, p_1^\lambda(z,w)^{-1}  \Psi_\mu^\lambda \ast_\lambda p_{1/2}^\lambda(z,w).$$
By taking the Weyl transform on both sides of  the equation
$$  \delta_\lambda^{|\mu|} \, c_\mu\,  (c_{\lambda/2}x-i s_{\lambda/2}u, c_{\lambda/2} u+i s_{\lambda/2} x)^\mu \, p_1^\lambda(x,u) = \pi^{-n/2} \sqrt{c_{n,\lambda}}\, \, \Psi_\mu^\lambda \ast_\lambda p_{1/2}^\lambda(x,u).$$
 we obtain the relation
$$ \pi^{n/2} c_{n,\lambda}^{-1/2}\,c_\mu\, \delta_\lambda^{|\mu|}  \pi_\lambda(  (x,u)_\lambda^\mu \, p_1^\lambda(x,u)) = \pi_\lambda( \Psi_\mu^\lambda) \, e^{-\frac{1}{2}H(\lambda)}.$$
In the above, we have written $ (x,u)_\lambda$ as a short hand for $(c_{\lambda/2}x-i s_{\lambda/2}u, c_{\lambda/2} u+i s_{\lambda/2} x).$\\

We now make use of the following relations stated in Proposition 2.1.  For suitable $ f \in L^2(\R^{2n}) $  we have
$$  \pi_\lambda( \lambda x_jf) = i \partial_{P_j} \pi_\lambda(f),\,\,\,\, \pi_\lambda(  u_jf) =\partial_{Q_j} \pi_\lambda(f).$$
Without loss of generality we assume $ \lambda > 0 $ from now on and suppress the dependence on $ \lambda $ for the sake of simplicity of notation. For $ 1 \leq j \leq n, $ define 
$$ D_j =  i (c_{\lambda/2}\partial_{P_j}- \lambda\, s_{\lambda/2}\partial_{Q_j}) \,\,\,  D_{j+n} = (\lambda c_{\lambda/2} \partial_{Q_j}- s_{\lambda/2} \partial_{P_j})$$ and let $ D^\mu = D_1^{\mu_1} D_2^{\mu_2} \cdots D_{2n}^{\mu_{2n}}.$ Since $ \pi_\lambda(p_t^\lambda) = e^{-tH(\lambda)} $ it is clear that  $ \pi_\lambda( \Psi_\mu^\lambda) $ satisfies the equation
$$  \pi_\lambda( \Psi_\mu^\lambda) \, e^{-\frac{1}{2}H(\lambda)} = c_\mu^\lambda\, (\lambda  s_\lambda)^{-|\mu|/2} \, D^\mu e^{-H(\lambda)},\,\,\, c_\mu^\lambda = \pi^{n/2} c_{n,\lambda}^{-1/2}\,c_\mu\,.$$
The  similarity between the above equation and the definition of the Hermite functions  suggest that the operators $ S_\mu^\lambda = \pi_\lambda( \Psi_\mu^\lambda) $ can be considered as the operator analogues of Hermite functions. For the lack of a better name, we would like to call them Quantum Hermite functions. The following theorem further strengthens  this analogy.

\begin{thm} Let $ \mathcal{S}_2 $ be the Hilbert space of Hilbert-Schmidt operators on $ L^2(\R^n) $ equipped with the inner product $ (S, T) = \tr (ST^\ast).$ Then the family of operators $ S_\mu^\lambda, \mu \in \mathbb N^{2n} $ forms as orthonormal basis for $ \mathcal{S}_2.$
\end{thm}
\begin{proof}  We have proved that  $ T_\lambda : \mathcal{F}(\C^{2n}) \rightarrow \Fs $ is an onto isometry and so  it follows that
$$ \int_{\C^{2n}} T_\lambda F(\zeta)\, \overline{T_\lambda G(\zeta)}\, w_\lambda(\zeta)\, d\zeta = \int_{\C^{2n}}  F(\zeta)\, \overline{ G(\zeta)}\, w_0(\zeta)\, d\zeta.$$
Since $ \zeta_\mu,\, \mu \in \mathbb N^{2n} $ is an orthonormal basis for $ \mathcal{F}(\C^{2n}),$ it is clear that $ T_\lambda \zeta_\mu, \mu \in \mathbb N^{2n} $ forms an orthonormal basis for $ \Fs.$ 
Further, $ B_\lambda : L^2(\R^{2n}) \rightarrow  \Fs $ and $ \pi_\lambda : L^2(\R^{2n}) \rightarrow \mathcal{S}_2 $ are unitary. Hence $ \Psi_\mu^\lambda,\, \mu \in \mathbb N^{2n} $ is an orthonormal basis for $ L^2(\R^{2n}) $ and hence $ S_\mu^\lambda,\, \mu \in \mathbb N^{2n} $ turns out to be an orthonormal basis for $\mathcal{S}_2 .$
\end{proof} 

We proceed to bring out further analogies between Hermite functions and their operator analogues .  
Recall that from the equation $ \Phi_\mu(\xi) e^{-\frac{1}{2}|\xi|^2} = (-1)^{|\mu|}\,c_\mu\, \partial_\xi^\mu e^{-|\xi|^2} $ which defines the normalised Hermite functions $ \Phi_\mu $  on $ \R^{2n},$  we get  the relation 
$$  ( -\frac{\partial}{\partial \xi_j}+\xi_j)\Phi_\mu(\xi) = \sqrt{2\mu_j+2}\, \Phi_{\mu+e_j}(\xi) .$$
We have an analogue of this for the operators $ S_\mu^\lambda.$  Note that as  $ [P_j,Q_j] = I,$ it follows that $ D_j $ is a commuting family of derivations. Applying the derivation $ D_j $ to the defining relation
 $ S_\mu^\lambda e^{-\frac{1}{2}H(\lambda)} = c^\lambda_\mu\, (\lambda \sinh \lambda)^{-|\mu|/2}  \, D^\mu e^{-H(\lambda)}$ we obtain 
$$  D_j S_\mu^\lambda\, e^{-\frac{1}{2}H(\lambda)} + S_\mu^\lambda\, D_j  e^{-\frac{1}{2}H(\lambda)}= \sqrt{\lambda \sinh \lambda}\, \sqrt{2\mu_j+2}\, \,S_{\mu+e_j}^\lambda e^{-\frac{1}{2}H(\lambda)}.$$ 
With  $ M_j= M_j(\lambda) = (D_j  e^{-\frac{1}{2}H(\lambda)}) e^{\frac{1}{2}H(\lambda)},$ we define $ \mathcal{A}_j T = D_j T+ T M_j.$ Then we have 
\begin{equation}\label{create} \mathcal{A}_j S_\mu^\lambda\, = \sqrt{\lambda \sinh \lambda}\, \sqrt{2\mu_j+2}\, \,S_{\mu+e_j}^\lambda .
\end{equation}
We can write down the operators $ M_j $ explicitly by calculating their actions on the Hermite functions $ \Phi_\alpha^\lambda$ on $ \R^n.$ \\

Recall the annihilation and creation operators  $ A_j = \frac{\partial}{\partial \xi_j}+ \lambda \xi_j $ and $ A_j^\ast =-\frac{\partial}{\partial \xi_j}+ \lambda \xi_j .$ Expressing $ P_j $ and $ Q_j $ in terms of $ A_j $ and $A_j^\ast $ we see that
$$  D_j =  \frac{i}{2} (e^{-\lambda/2}\partial_{A_j}- \, e^{\lambda/2}\partial_{A_j^\ast}),\,\,\,\,  D_{j+n} = \frac{1}{2}(e^{\lambda/2}\partial_{A_j^\ast} +\, e^{-\lambda/2}\partial_{A_j}).$$ 
The calculations done in Lemma \ref{lem.2.4} shows that 
$$  \frac{1}{2} e^{-\lambda/2}\, \partial_{A_j}e^{-\frac{1}{2}H(\lambda)}  =  s_{\lambda/2} \, A_j  \,e^{-\frac{1}{2}H(\lambda)},\,\,\, \frac{1}{2} e^{\lambda/2} \,\partial_{A_j^\ast} \,e^{-\frac{1}{2}H(\lambda)} = 
-s_{\lambda/2} \, A_j^\ast  \,e^{-\frac{1}{2}H(\lambda)}.$$
In view of these, it follows that for $ 1 \leq j \leq n,$
$$  D_j e^{-\frac{1}{2}H(\lambda)} = 2i \, s_{\lambda/2} \, \lambda \, Q_j e^{-\frac{1}{2}H(\lambda)},\,\,\, D_{j+n} e^{-\frac{1}{2}H(\lambda)} = 2\,   s_{\lambda/2}\,  P_je^{-\frac{1}{2}H(\lambda)}.$$
This proves that $ M_j = 2i \, s_{\lambda/2}\,  \lambda\,  Q_j $ and $ M_{j+n} = 2\,   s_{\lambda/2}\, P_j.$  Thus we have the following result.

\begin{prop}  For $ 1 \leq j \leq n,$  consider the following operators defined on $ \mathcal{S}_2:$
$$ \mathcal{A}_j T = D_j T+ 2i \, s_{\lambda/2}\,  \lambda\, T Q_j,\,\,\, \mathcal{A}_{j+n} T = D_{j+n} T+ 2\, s_{\lambda/2} \,T P_j.$$
Let $ \mathcal{A}_j ^\ast$ stand for the formal adjoint of $  \mathcal{A}_j $ on $ \mathcal{S}_2.$ Then we have
$$ \mathcal{A}_j S_\mu^\lambda \, = \sqrt{\lambda \sinh \lambda}\, \sqrt{2\mu_j+2}\, S_{\mu+e_j}^\lambda, \,\,\,  \mathcal{A}_j^\ast S_\mu^\lambda \, = \sqrt{\lambda \sinh \lambda}\, \sqrt{2\mu_j}\, S_{\mu-e_j}^\lambda .$$ 
Consequently, $ S_\mu^\lambda $ are eigenvectors of the operator $ \mathcal{H} = \frac{1}{2} \sum_{j=1}^{2n} \left(\mathcal{A}_j\mathcal{A}_j^\ast  +\mathcal{A}_j^\ast \mathcal{A}_j\right)$
  with eigenvalues $ (\lambda \sinh \lambda) \, (2|\mu|+2n).$
\end{prop}
\begin{proof} The claim about $ \mathcal{A}_j$ follows from \eqref{create}. To prove the same for $ \mathcal{A}_j ^\ast $ consider the expansion of $ \mathcal{A}_j ^\ast S_\mu^\lambda $ in terms of  $ S_{\mu^\prime}^\lambda:$
$$ \mathcal{A}_j ^\ast S_\mu^\lambda = \sum_{\mu^\prime \in \mathbb N^{2n}}  ( \mathcal{A}_j ^\ast S_\mu^\lambda, S_{\mu^\prime}^\lambda)\, S_{\mu^\prime}^\lambda = \sum_{\mu^\prime \in \mathbb N^{2n}}  (S_\mu^\lambda, \mathcal{A}_j S_{\mu^\prime}^\lambda)\, S_{\mu^\prime}^\lambda.$$
The orthonormality of the system $ S_{\mu^\prime}^\lambda, \mu^\prime \in \mathbb N^{2n} $ together with the result for $ \mathcal{A}_j $ proves our claim on $ \mathcal{A}_j^\ast.$ The claim on $ \mathcal{H} $ is a consequence of the result on $\mathcal{A}_j$ and its adjoint.
\end{proof}  

\begin{rem} From the result of Lemma \ref{lem.2.4} we can easily calculate $ S_{e_j}^\lambda,$ where $ e_j $ are the coordinate vectors in $ \R^{2n}.$ Indeed, for $ 1 \leq j \leq n,$ we have
$$ D_j e^{-H(\lambda)} = i \sinh \lambda \left(e^{\lambda/2} A_j + e^{-\lambda/2} A_j^\ast \right) e^{-H(\lambda)},$$
$$D_{j+n}e^{-H(\lambda)} =  \sinh \lambda \left(e^{\lambda/2} A_j - e^{-\lambda/2} A_j^\ast \right) e^{-H(\lambda)}.$$
Thus $ S_{e_j}^\lambda $ are of the form $ P(A_j, A_j^\ast)\, e^{-\frac{1}{2}H(\lambda)} $ where $ P $ is a first degree polynomial in two variables. We can also say that $ S_{e_j}^\lambda $ is a linear combination of  operators of the form $ R_j = A_j e^{-\frac{1}{2}H(\lambda)} $ and $ R_j ^\ast= A_j^\ast e^{-\frac{1}{2}H(\lambda)} .$ Since $ D_j, 1 \leq j \leq 2n $ is a commuting family of derivations we can easily show by iteration that for any $ \mu \in \mathbb N^{2n},$ there exist polynomials $ P_\mu(x,u) $ in $2n$ variables such that $ S_\mu^\lambda = P_\mu( A, A^\ast) e^{-\frac{1}{2}H(\lambda)}$ where $ A = (A_j) $ and $ A^\ast = (A_j^\ast).$
\end{rem}

It can be shown that  $ S_\mu^\lambda = P_\mu^\lambda(A,A^\ast) e^{-\frac{1}{2}H(\lambda)} $  are pseudo-differential operators of order zero and hence  bounded on $ L^p(\R^n) $ for all  $ 1 < p < \infty .$ There is yet another way of seeing this. We can  use a transference theorem of Mauceri \cite{GM}. According to this theorem if  $ M \in B(L^2(\R^n))$  is such that the Weyl multiplier $ \widetilde{T}_M $ defined by $ \pi_\lambda(\widetilde{T}_M f) = \pi_\lambda(f)M $ is bounded on $ L^p(\R^{2n})$ then $ M $ is bounded on $ L^p(\R^n).$   Since $ \pi_\lambda(\Psi_\mu^\lambda) = S_\mu^\lambda,$ it is clear that $ f  \rightarrow f \ast_\lambda \Psi_\mu^\lambda $ is  a right Weyl multiplier. Moreover, this operator is bounded on $ L^p(\R^{2n})$ for all $ 1 \leq p \leq \infty.$ Hence we have the following result.

\begin{prop}  For any $ \mu \in \Na^{2n}, \,S_\mu^\lambda $ is bounded on $ L^p(\R^n)$ for all $ 1 < p < \infty.$ Moreover,  $S_\mu^\lambda $ belongs to the Schatten-von Neumann class $ \mathcal{S}_p$ for any $ 1\leq p < \infty.$
\end{prop}
\begin{proof} Since $ S_\mu^\lambda \in \mathcal{S}_2,$ we only need to show that it is of trace class. For, once we know this, we can use interpolation and duality to prove the rest. Using the Hermite basis $ \Phi_\alpha^\lambda, \alpha \in \Na^n,$  we see that $ |( S_\mu^\lambda \Phi_\alpha^\lambda, \Phi_\alpha^\lambda)| = O( (2|\alpha|+n)^{|\mu|/2}) e^{-\frac{1}{2}(2|\alpha+n)|\lambda|}.$ This proves that $ S_\mu^\lambda \in \mathcal{S}_1.$
\end{proof}

\begin{rem} The operators $ P_\mu = P_\mu(A,A^\ast) $ are the operator analogues of the (normalised) Hermite polynomials defined by the relation $ \Phi_\alpha(x) = H_\alpha(x) e^{-\frac{1}{2}|x|^2}.$ Recall that $ H_\alpha, \alpha \in \mathbb N^n,$ forms an orthonormal basis for  $ L^2(\R^n, d\gamma)$ where $  d\gamma(x) = e^{-|x|^2}\,dx $ is the Gaussian measure. Let $ \mathcal{S}_2(\gamma) $ stand for the subspace of  linear operators $ T $ on $ L^2(\R^n) $ such that $ T e^{-\frac{1}{2}H(\lambda)} \in \mathcal{S}_2.$ We equip $ \mathcal{S}_2(\gamma) $ with the inner product 
$$ (T,S)_\gamma =  (T e^{-\frac{1}{2}H(\lambda)}, S e^{-\frac{1}{2}H(\lambda)})= \tr ( e^{-H(\lambda)}  S^\ast T) $$ 
which turns it into a Hilbert space. It is then clear that $ (P_\mu, P_\nu)_\gamma = (S_\mu^\lambda, S_\nu^\lambda) $
and hence $ P_\mu, \mu \in \mathbb N^{2n},$ forms an orthonormal basis for $ \mathcal{S}_2(\gamma).$
\end{rem}

\begin{rem} Let $ L^2_\gamma(\R^{2n}) $ stand for the space of all tempered distributions $ f $ on $ \R^{2n}$ such that  $ f \ast_\lambda p_{1/2}^\lambda \in L^2(\R^{2n}).$ Then it follows that  $ \pi_\lambda : L^2_\gamma(\R^{2n}) \rightarrow  \mathcal{S}_2(\gamma) $ sets up a one-to-one correspondence. Consequently, for any $ f \in L^2_\gamma(\R^{2n}),$ we have the expansion
$$ \pi_\lambda(f) = \sum_{\mu \in \mathbb N^{2n}} (\pi_\lambda(f), P_\mu)\, P_\mu $$
where the series converges in $ \mathcal{S}_2(\gamma).$
\end{rem}

\subsection{Deformed Hermite functions} We defined the operators $ S_\mu^\lambda $ as the Weyl transform $ \pi_\lambda(\Psi_\mu^\lambda)$ and so it would be interesting see how these functions
look like. Recall that $ \Psi_\mu^\lambda $ are defined by the equation 
$$ \Psi_\mu^\lambda \ast_\lambda p_{1/2}^\lambda(x,u) = \pi^{n/2} c_{n,\lambda}^{-1/2}\,\zeta_{\mu}(\sigma_\lambda(x,u))\, p_1^\lambda(x,u).$$
This is the exact analogue of an equation satisfied by the Hermite functions.  Indeed, let $ q_t $ stand for the heat kernel associated to the Laplacian on $ \R^n$ explicitly given by 
$$ q_t(x) = (4\pi t)^{-n/2} \, e^{-\frac{1}{4t}|x|^2} .$$ 
By taking the Fourier transform of the  defining relation for the Hermite functions 
$$   \Phi_\alpha(x) e^{-\frac{1}{2}|x|^2} = (-1)^{|\alpha|}\, c_\alpha\, \partial_\alpha e^{-|x|^2}$$ 
and noting that $ \Phi_\alpha $ are eigenfunctions of the Fourier transform with eigenvalues $ (-i)^{|\alpha|} $ we get the equation
$$ \Phi_\alpha \ast q_{1/2}(\xi) = (2\pi)^{n/2} c_\alpha\, (-\xi)^\alpha q_1(\xi) .$$
Let us try to compute some of the functions $ \Psi_\mu^\lambda $ for small values of $ |\mu|.$\\

For $ \mu = 0 $ we have $ \Psi_0^\lambda(x,u) =  \pi^{n/2} c_{n,\lambda}^{-1/2}\ p_{1/2}^\lambda(x,u) $ because then
$$ \Psi_0^\lambda \ast_\lambda p_{1/2}^\lambda(x,u) = \pi^{n/2} c_{n,\lambda}^{-1/2}\,p_{1/2}^\lambda \ast_\lambda p_{1/2}^\lambda(x,u) = \pi^{n/2}c_{n,\lambda}^{-1/2}\,\, p_1^\lambda(x,u).$$
The first order differential operators $ Z_j^\lambda $ and $ \bar{Z}_j^\lambda $ given in \eqref{Z-j-action} are such that
$$  \pi_\lambda(Z_j^\lambda p_{1/2}^\lambda) = -i A_j e^{-\frac{1}{2}H(\lambda)}, \,\,\pi_\lambda(\bar{Z}_j^\lambda p_{1/2}^\lambda) =-i A_j ^\ast e^{-\frac{1}{2}H(\lambda)}.$$
Since $ \pi_\lambda(\Psi_{e_j}^\lambda) = S_{e_j}^\lambda $ which is explicitly given in the above remark, we see that for $ 1 \leq j \leq n,$
$$ \Psi_{e_j}^\lambda(x,u) = (\sinh \lambda) \left(e^{\lambda/2} Z_j^\lambda + e^{-\lambda/2} \bar{Z}_j^\ast \right) p_{1/2}^\lambda(x,u).$$
And a similar formula holds for $ \Psi_{e_{j+n}}^\lambda(x,u).$ By iteration we can easily prove that there exist polynomials  $ P_\mu $ such that
$ \Psi_\mu^\lambda(x,u) = P_\mu(x,u)\, p_{1/2}^\lambda(x,u).$ It is clear that this family of functions $ \Psi_\mu^\lambda, \mu \in \mathbb N^{2n} $ forms an orthonormal basis for $ L^2(\R^{2n}).$
It turns out that the functions $ \Psi_\mu^\lambda $ are none other than Hermite functions suitably rescaled (or deformed).\\

Let $ \Phi_\mu$ be the standard Hermite functions on $ \R^{2n}$ so that $ B \Phi_\mu = \zeta_\mu.$ Then by definition we have $  (B_\lambda^\ast \circ T_\lambda \circ B) \Phi_\mu = \Psi_\mu^\lambda.$ It turns out that the operator $ B_\lambda^\ast \circ T_\lambda \circ B$ is a constant multiple of a dilation operator.

\begin{prop}\label{tlambda-dlambda} For $ \lambda \neq 0, $ let  $ D_\lambda:L^2(\R^{2n}) \rightarrow L^2(\R^{2n})$ be  defined by $ D_\lambda f(x,u) = (c_{\lambda/2}\delta_\lambda)^{n} f (c_{\lambda/2}\delta_\lambda (x,u)).$ $($Note that $(c_{\lambda/2}\delta_\lambda)^{2} = (\lambda/2)(\coth \lambda/2)).$ Then $ B_\lambda^\ast \circ T_\lambda \circ B = D_\lambda$ so that  $ \Psi_{\mu}^\lambda = D_\lambda \Phi_\mu.$ The functions $ \Psi_\mu^\lambda $ are eigenfunctions of the deformed Hermite operator $ H_\lambda = -\Delta + (\lambda/2\, \coth \lambda/2)^2 |\xi|^2$ on $ \R^{2n} $ with eigenvalues $ (2|\mu|+2n) (\lambda/2\, \coth \lambda/2).$
\end{prop}
\begin{proof} Recalling the definition of the Bargmann transform on $ L^2(\R^{2n})$ we have
$$ (T_\lambda \circ B)f(z,w) =  \pi^{-n/2} \, \sqrt{c_{n,\lambda}}\, \,e^{\frac{1}{4} (Z^2+W^2)} \int_{\R^{2n}} f(x,u) \,e^{-\frac{1}{2}((x-Z)^2+(u-W)^2)} dx\, du$$
where $ (Z,W) =  \sigma_\lambda(z,w)= \delta_\lambda (c_{\lambda/2}z-i s_{\lambda/2}w, c_{\lambda/2} w+i s_{\lambda/2} z).$ By making the change of variables $ (x,u) \rightarrow c_{\lambda/2}\delta_\lambda (x,u)$ the above integral takes the form
$$ (T_\lambda \circ B)f(z,w) = \pi^{-n/2}  \sqrt{c_{n,\lambda}}\, (c_{\lambda/2}\delta_\lambda)^{n} \,e^{\frac{1}{4} (Z^2+W^2)} \int_{\R^{2n}} D_\lambda f(x,u) \,e^{-\frac{1}{2}((c_{\lambda/2}\delta_\lambda x-Z)^2+(c_{\lambda/2}\delta_\lambda u-W)^2)} dx\, du.$$
A simple calculation  shows that  $ e^{-\frac{1}{2}((c_{\lambda/2}\delta_\lambda x-Z)^2+(c_{\lambda/2}\delta_\lambda u-W)^2)} $ reduces to
$$ e^{-\frac{1}{2}\delta_\lambda^2 c_{\lambda/2}^2 ((x-z)^2+(u-w)^2)} \,  e^{-i c_{\lambda/2} s_{\lambda/2} \delta_\lambda^2 (w(x-z)-z(u-w))}\, e^{\frac{1}{2} \delta_\lambda^2 s_{\lambda/2}^2 (z^2+w^2)}.$$
As $ 2 s_{\lambda/2}c_{\lambda/2} = s_\lambda $ the above integral is nothing but
$$ e^{\frac{1}{2} \delta_\lambda^2 s_{\lambda/2}^2 (z^2+w^2)} \int_{\R^{2n}} D_\lambda f(x,u) \, e^{-\frac{\lambda}{4} (\coth \lambda/2) ((x-z)^2+(u-w)^2)} \,  e^{-i \frac{\lambda}{2} (w \cdot x-z \cdot u)}\,dx\, du.$$
We can also check that 
$$ (Z^2+W^2)+ 2\delta_\lambda^2 s_{\lambda/2}^2 (z^2+w^2) = \frac{\lambda}{\sinh \lambda}( 1+2 s_{\lambda/2}^2) (z^2+w^2)= \lambda (\coth \lambda)(z^2+w^2).$$ 
This proves that
$$ (T_\lambda \circ B)f(z,w) = \pi^{-n/2} (4\pi)^{n}  \sqrt{c_{n,\lambda}}\, (c_{\lambda/2}\delta_\lambda)^{n} \Big(\frac{ \sinh \lambda/2}{\lambda}\Big)^n\, c_{n,\lambda}\,p_1^\lambda(z,w)^{-1} \, D_\lambda f \ast_\lambda p_{1/2}^\lambda(z,w). $$
After simplification the above becomes  $ ( T_\lambda \circ B)f(z,w) =   (B_\lambda \circ D_\lambda)f(z,w)$ proving the result. 
As $ B \Phi_\mu(z,w) = \zeta_\mu(z,w),$ it follows that 
$$ T_\lambda \zeta_\mu(z,w)= 2^{-n}(4\pi)^{n/2} \sqrt{c_{n,\lambda}}\, \zeta_\mu (\sigma_\lambda(z,w)) = c_{n,\lambda}\, \, p_1^\lambda(z,w)^{-1}  D_\lambda \Phi_\mu \ast_\lambda p_{1/2}^\lambda(z,w)$$
which is the defining relation for $ \Psi_\mu^\lambda.$ This proves that $ \Psi_\mu^\lambda = D_\lambda \Phi_\mu.$ Since $ \Phi_\mu $ are eigenfunctions of $ H= -\Delta +|\xi|^2 $ with eigenvalues $ (2|\mu|+2n)$ it follows that $ \Psi_\mu^\lambda = D_\lambda \Phi_\mu $ are eigenfunctions of $ H_\lambda $ as claimed.  
\end{proof}

As we have identified the functions $ \Psi_\mu^\lambda $ as deformed Hermite functions, it is now easy to write down a generating function for the operators $ S_\mu^\lambda.$

\begin{prop}\label{prop-gen-id} The operators $ S_\mu^\lambda $ satisfy the generating function identity
$$\sum_{\mu \in \Na^{2n}}  (-i)^{|\mu|}\,  \zeta_\mu(z,w) S_\mu^\lambda = \pi^{-n/2} c_{n,\lambda}^{-1/2}\,e^{\frac{1}{4} (z^2+w^2)}  G_\lambda(r_\lambda(z,w)) $$
where $ G_\lambda(z,w) = \pi_\lambda(w,-z) e^{-\frac{1}{2}H(\lambda)} \pi_\lambda(-w,z) $ and $  \lambda^2 r_\lambda^2  = (\lambda/2)\,\coth(\lambda/2).$
\end{prop}
\begin{proof} From the generating function in   \eqref{gen-id} satisfied by the Hermite functions we obtain
$$  \sum_{\mu \in \Na^{2n}} \Psi_\mu^\lambda(x,u)\, \zeta_\mu(z,w)  = \pi^{-n} (c_{\lambda/2}\delta_\lambda)^n  e^{-\frac{1}{4} ( z^2+w^2)} e^{ c_{\lambda/2}\delta_\lambda(x,u) \cdot (z,w)}\, e^{- \frac{1}{2} (\lambda/2) \coth (\lambda/2) (|x|^2+|u|^2)}.$$
After simplifying the right hand side of the above identity becomes 
$$ 2^n \lambda^{-n/2} (\sinh \lambda)^{n/2}\,e^{-\frac{1}{4} (z^2+w^2)} e^{ c_{\lambda/2}\delta_\lambda(x,u) \cdot (z,w)} \,p_{1/2}^\lambda(x,u).$$
Therefore, the operators $ S_\mu^\lambda $  satisfy the generating function identity
$$\sum_{\mu \in \Na^{2n}}  \zeta_\mu(z,w) S_\mu^\lambda = \pi^{-n/2} c_{n,\lambda}^{-1/2}\,e^{-\frac{1}{4} (z^2+w^2)} \int_{\R^{2n}} e^{ c_{\lambda/2}\delta_\lambda(x,u) \cdot (z,w)} \,p_{1/2}^\lambda(x,u) \pi_\lambda(x,u) dx du.$$
By defining $ \lambda^2\,r_\lambda^2 =  c_{\lambda/2}^2 \, \delta_\lambda^2 = (\lambda/2) \coth (\lambda/2)$ and replacing $ (z,w) $ by $(iz,iw) $ we can rewrite the above as 
$$\sum_{\mu \in \Na^{2n}}  (-i)^{|\mu|}\,  \zeta_\mu(z,w) S_\mu^\lambda = \pi^{-n/2} c_{n,\lambda}^{-1/2}\,e^{\frac{1}{4} (z^2+w^2)} \int_{\R^{2n}} e^{ -i \lambda r_\lambda(x,u) \cdot (z,w)} \,p_{1/2}^\lambda(x,u) \pi_\lambda(x,u) dx du.$$
The operator appearing on the right hand side of the above equation can be written down in a better way. 

Recall that in terms of  the Schr\"odinger representation $ \pi_\lambda $ of the Heisenberg group, we have  $ \pi_\lambda(x,u) = \pi_\lambda(x,u,0),$ and hence an  easy calculation 
shows that
$$ \pi_\lambda(a,b) \pi_\lambda(x,u) \pi_\lambda(-a,-b) = e^{i\lambda (x \cdot b- u \cdot a)} \, \pi_\lambda(x,u).$$
The unitary operators $ \pi_\lambda(a,b) $ can also be defined for complex values of $ a, b $ as unbounded operators. Therefore, from the above identity we obtain
$$ \pi_\lambda(w,-z) \pi_\lambda(x,u) \pi_\lambda(-w, z) = e^{ -i \lambda (x \cdot z+ u \cdot  w)} \, \pi_\lambda(x,u).$$
Consequently, since $ \pi_\lambda(p_{1/2}^\lambda) = e^{-\frac{1}{2}H(\lambda)}$ we have
\begin{equation}\label{glambda} \int_{\R^{2n}} e^{-i\lambda (x\cdot z+u \cdot w)} \,p_{1/2}^\lambda(x,u) \pi_\lambda(x,u) dx du = \pi_\lambda(w,-z) e^{-\frac{1}{2}H(\lambda)} \pi_\lambda(-w,z).
\end{equation}
This suggests that we consider the operator valued function 
$$ G_\lambda(z,w) = \pi_\lambda(w,-z) e^{-\frac{1}{2}H(\lambda)} \pi_\lambda(-w,z) $$
which can be considered as the translate of the operator $ e^{-\frac{1}{2}H(\lambda)}.$ 
In terms of $ G_\lambda $  the generating function identity for the operators $ S_\mu^\lambda $ takes the compact form
$$\sum_{\mu \in \Na^{2n}}  (-i)^{|\mu|}\,  \zeta_\mu(z,w) S_\mu^\lambda = \pi^{-n/2} c_{n,\lambda}^{-1/2}\,e^{\frac{1}{4} (z^2+w^2)}  G_\lambda(r_\lambda(z,w)).$$
This completes the proof of the proposition.
\end{proof}



\subsection{Sobolev spaces and Schwartz class operators} In \cite{KKR} the authors have introduced the operator analogues of Schwartz functions which they call Schwartz operators.  Since a smooth function $ f $ on $ \R^n$  is defined to be Schwartz if $ x^\alpha\, \partial^\beta f $ are bounded for all $ \alpha, \beta \in \mathbb N^n.$  The operator analogue is simply defined by replacing $ x^\alpha $ and $ \partial^\beta $ by $ Q^\alpha $ and $ P^\beta $ where $ Q =(Q_j) $ and $ P=(P_j) $ are the vectors of position and momentum operators. Thus according to \cite{KKR} a bounded linear operator $ T $ acting on  $   L^2(\R^n) $ is said to be Schwartz if $ P^\alpha Q^\beta T P^{\alpha^\prime} Q^{\beta^\prime} $ are all bounded. The class of all Schwartz operators is denoted by $ \mathcal{S}(\mathcal{H}).$ In \cite{KKR} the authors have established several properties  of Schwartz operators.\\

Using the operators $ S_\mu^\lambda $ we can give a simple description of $ \mathcal{S}(L^2(\R^n)) $ and its dual $ \mathcal{S}(L^2(\R^n))^\prime $ consisting of tempered operators. To motivate our definition recall that the Schwartz space $ \mathcal{S}(\R^n) = \cap_{s \geq 0} W_H^s(\R^n),$ the intersection of all the Hermite-Sobolev spaces. Here $ W_H^s( \R^n) $ is the image of $ L^2(\R^n) $ under the fractional power $ H(\lambda)^{-s/2} $ of the Hermite operator. In other words, for $ f \in W_H^s(\R^n), s \geq 0,$ if and only if $ f = H(\lambda)^{-s/2}g $ for some $ g \in L^2(\R^n).$ Equivalently, $ f \in W_H^s(\R^n)$ if and only if 
$$  \|f\|_{(s)}^2 = \, \sum_{\alpha \in \mathbb N^n}  ( (2|\alpha|+n)|\lambda|)^s\, |(f, \Phi_\alpha^\lambda)|^2 < \infty.$$
From this definition it follows that $ f \in \mathcal{S}(\R^n)$ if and only if  for every $ k \in \mathbb N, $ 
$$  | (f, \Phi_\alpha^\lambda)| \leq c_{k} ((2|\alpha|+n)|\lambda|)^{-k} ,\,\,\,  \alpha \in \mathbb N^n.$$
By duality, we also know that  a distribution $ f $  is tempered if and only if 
for some $ k \in \mathbb N, $ we have the estimates
$$  | (f, \Phi_\alpha^\lambda)| \leq c_{k} ((2|\alpha|+n)|\lambda|)^{k} ,\,\,\,  \alpha \in \mathbb N^n.$$
Since  $ S_\mu^\lambda $ are the operator analogues of  $ \Phi_\alpha^\lambda,$ the above motivates the following definitions.\\

For   $ s \geq 0,$ we say that $ T \in W_\mathcal{H}^s $ if and only if 
$$  \|T\|_{(s)}^2 = \, \sum_{\mu \in \mathbb N^{2n}}  ( (2|\mu|+2n)|\lambda|)^s\, |(T, S_\mu^\lambda)|^2 < \infty.$$
When  $ s< 0 $ we define  $W_\mathcal{H}^{-s} $ as the dual of $ W_\mathcal{H}^s.$ Then the space $ \mathcal{S}(L^2(\R^n)) $ defined in \cite{KKR} has the following description.

\begin{thm} We have  the equality $ \mathcal{S}(L^2(\R^n)) = \cap_{s \geq 0} W_\mathcal{H}^s,$ and consequently  $ T \in \mathcal{S}(L^2(\R^n)) $ if and only if
for every $ k \in \mathbb N, $ 
$$  | (T, S_\mu^\lambda)| \leq c_{k} ((2|\mu|+2n)|\lambda|)^{-k} ,\,\,\,  \mu \in \mathbb N^{2n}.$$
If we let $ \mathcal{S}(L^2(\R^n))^\prime $ stand for the dual of  $ \mathcal{S}(L^2(\R^n)),$ then $ \Lambda \in \mathcal{S}(L^2(\R^n))^\prime $ if and only if 
for some $ k \in \mathbb N, $ we have the estimates
$$  | (\Lambda,  S_\mu^\lambda)| \leq c_{k} ((2|\mu|+2n)|\lambda|)^{k} ,\,\,\,  \mu \in \mathbb N^{2n}.$$
\end{thm}
\begin{proof} In \cite{KKR} the authors have already proved that $ T \in \mathcal{S}(L^2(\R^n))$ if and only if $ T = \pi_\lambda(f) $ where $ f \in \mathcal{S}(\R^{2n}).$  By expanding $ f $ in terms of $ \Psi_\mu^\lambda = D_\lambda \Phi_\mu $  it is easily seen that $ f \in \mathcal{S}(\R^{2n})$ if and only if $ (f, \Psi_\mu^\lambda) $ satisfies the required decay estimates. Since $ (T, S_\mu^\lambda) = (f, \Psi_\mu^\lambda) $ the first part of the theorem follows. We remark that the topology of the space $ \mathcal{S}(\R^{2n}) $ can  be described in terms of the  Hermite-Sobolev norms $ \|f\|_{(k)} $ and hence the topology of $ \mathcal{S}(L^2(\R^n)) $ can also be described in terms of  $ \|T\|_{(k)}.$

Given $ \Lambda \in \mathcal{S}(L^2(\R^n))^\prime $ let us define a  linear functional on $ \mathcal{S}(\R^{2n})$ by the prescription $ \varphi_\Lambda(f) = \Lambda(\pi_\lambda(f)).$ Since $ |\Lambda(T)| \leq C \|T\|_{(m)} $ for some $ m $ and all $ T \in \mathcal{S}(L^2(\R^n)),$ it follows that $ \varphi_\Lambda $ is a tempered distribution. Conversely, to a given tempered distribution $ g $  we can define a functional $ \Lambda_g $  on $ \mathcal{S}(L^2(\R^n))$ by setting $ \Lambda_g(\pi_\lambda(f)) = (g,f).$ As $ g $ is tempered, $\Lambda_g \in  \mathcal{S}(L^2(\R^n))^\prime .$ Thus there is a one to one correspondence between $ \mathcal{S}(L^2(\R^n))^\prime $ and 
$ \mathcal{S}(\R^{2n})^\prime .$ This together with the characterisation of the tempered distributions in terms of their action on Hermite functions complete the proof of the second part of the theorem.
\end{proof}
\subsection{Fourier transform of operators} Recall that in the classical setting the Bargmann transform $ B $ and the Fourier transform on $ L^2(\R^n) $ are related via the unitary operator $ U : \mathcal{F}(\C^n) \rightarrow \mathcal{F}(\C^n) $ defined by $ UF(z) = F(-iz).$ These three operators are related by $ U\circ B = B\circ \mathcal{F} $ and hence the Fourier transform  can also be defined by $ \mathcal{F} = B^\ast \circ U \circ B.$ This motivates us to define  Fourier transform of operators as follows.  The operator $ U : \Fs \rightarrow \Fs $ defined by $ UF(z,w) = F(-iz,-iw) $ is also unitary for any $ \lambda.$ Hence it makes sense to define  the Fourier transform on $ \mathcal{S}_2 $ by following prescription. For any $ T \in \mathcal{S}_2, $ we define $ \widehat{T} $ to be the unique operator in $ \mathcal{S}_2 $ such that $ \mathcal{G}_\lambda(\widehat{T}) =  (U \circ \mathcal{G}_\lambda)(T).$ \\

\begin{thm} The Fourier transform $ \mathcal{F}: \mathcal{S}_2 \rightarrow \mathcal{S}_2, T \rightarrow \widehat{T} $ is a unitary operator. The operators $ S_\mu^\lambda $ are eigenvectors of the Fourier transform: $ \widehat{S}_\mu^\lambda = (-i)^{|\mu|} S_\mu^\lambda.$ Moreover, for any $ T \in \mathcal{S}_2,$ we have
$$ \widehat{T} = \sum_{\mu \in \mathbb N^{2n}}  (-i)^{|\mu|}\, (T, S_\mu^\lambda)\, S_\mu^\lambda.$$
For any $ T \in \mathcal{S}_p, 1\leq p \leq 2$ we have $ \| \widehat{T}\|{p^\prime} \leq \|T\|_{\mathcal{S}_p} $ where $ 1/p + 1/{p^\prime} = 1.$
\end{thm}
\begin{proof} If we let $ T = \pi_\lambda(f), f \in L^2(\R^{2n}),$ so that $ \mathcal{G}_\lambda(T) = B_\lambda(f) ,$ then $ \mathcal{G}_\lambda(\widehat{T}) = B_\lambda (U_\lambda f)$ where
where $ U_\lambda : L^2(\R^{2n}) \rightarrow L^2(\R^{2n}) $ is characterised by the property
$$ U_\lambda f \ast_\lambda p_{1/2}^\lambda(z,w) =  c_\lambda \, e^{-\frac{\lambda}{2}(z^2+w^2)}\, f \ast_\lambda p_{1/2}^\lambda(-iz,-iw).$$
In \cite{GT} we have shown that $ U_\lambda f(x,u) = c(\lambda)^n \, \widehat{f}(c(\lambda)(x,u)) $ where $ c(\lambda)= (\lambda/2) \coth (\lambda/2).$ Since $ \Phi_\mu $ are eigenfunctions of the Fourier transform and 
$$ \Psi_\mu^\lambda(x,u) = D_\lambda \Phi_\mu(x,u) = \sqrt{c(\lambda)}^n \Phi_\mu(\sqrt{c(\lambda)}(x,u))$$
it follows that $U_\lambda \Psi_\mu^\lambda= (-i)^{|\mu|}\, \Psi_\mu^\lambda. $ This proves that $ S_\mu^\lambda $ are eigenfunctions of $ \mathcal{F}.$ The expansion for $ \widehat{T} $ is immediate. The monotonicity of the Schatten-von Neumann norms proves the $ \mathcal{S}_p -  \mathcal{S}_{p^\prime}$ inequality for the Fourier transform.
\end{proof}

Our definition of the Fourier transform depends on the expansion of $ T $ in terms of $ S_\mu^\lambda .$ It  is preferable to have an integral representation of the Fourier transform at least for some class of operators. Let us review the classical situation in order to motivate our alternate definition. The Hermite functions $ \Phi_\alpha $ on $ \R^n$ are given by the generating function identity 
$$ \sum_{\alpha \in \mathbb N^n} (-i)^{|\alpha|}\, \zeta_\alpha(w)\, \Phi_\alpha(x) = \pi^{-n/2} e^{-\frac{1}{4}w^2} e^{-\frac{1}{2}(x+iw)^2}. $$
Given $ f \in L^2(\R^n) $ let $ F = Bf $ be its Bargmann transform which belongs to the Fock space $ \mathcal{F}(\C^n).$  Since $ (f, \Phi_\alpha) = ( F, \zeta_\alpha) $ we have
$$ \sum_{\alpha \in \mathbb N^n} (-i)^{|\alpha|}\, (f, \Phi_\alpha) \, \Phi_\alpha(x) = \pi^{-n/2}  \int_{\C^n} e^{-\frac{1}{4}\bar{w}^2} e^{-\frac{1}{2}(x-i\bar{w})^2}\, F(w) e^{-\frac{1}{2}|w|^2} \, dw. $$
Since the left hand side is nothing  but the Fourier transform of  $ f \in L^2(\R^n) $ we can take the right hand side as an equivalent definition of the Fourier transform. Note that $ Bf $ makes sense as an entire function for any $ f \in L^1(\R^n) $ and hence we can define the Fourier transform on $ L^1(\R^n) $ by the expression
$$ \widehat{f}(x) = \pi^{-n/2}  \int_{\C^n} e^{-\frac{1}{4}\bar{w}^2} e^{-\frac{1}{2}(x-i\bar{w})^2}\, Bf(w) e^{-\frac{1}{2}|w|^2} \, dw. $$
Writing down the definition of $ Bf $ and simplifying we can verify that the right hand side reduces to the usual definition of  $ \widehat{f}(x).$
In analogy with the above definition of the Fourier transform, we can make use of the generating function identity satisfied by $ S_\mu^\lambda $ to define Fourier transforms of operators.\\

Consider the twisted Gauss-Bargmann transform defined  for $ T \in \mathcal{S}_2 $ by the equation 
$$ \mathcal{G}_\lambda (T)(z,w) = (2\pi)^{-n/2}\, |\lambda|^{n/2}\,  c_{n,\lambda} \,\, p_1^\lambda(z,w)^{-1}\, \tr\left( \pi_\lambda(-z,-w) T e^{-\frac{1}{2}H(\lambda)} \right) $$
which coincides with $ B_\lambda f $ when $ T = \pi_\lambda(f).$  From the generating function identity proved in Proposition \ref{prop-gen-id} we have
\begin{equation}\label{def-g-lambda}
\sum_{\mu \in \Na^{2n}}  (-i)^{|\mu|}\,  \zeta_\mu^\lambda(z,w) S_\mu^\lambda = \pi^{-n/2} c_{n,\lambda}^{-1/2}\,e^{\frac{1}{4} (\lambda/s_\lambda)(z^2+w^2)} \,  \widetilde{G}_\lambda(z,w)
\end{equation}
where we have defined $ \widetilde{G}_\lambda(z,w) =G_\lambda(r_\lambda \sigma_\lambda(z,w)) $ and $ \zeta_\mu^\lambda(z,w)=  T_\lambda\zeta_\mu(z,w) .$  More explicitly, we have 
$$  \widetilde{G}_\lambda(z,w) =G_\lambda( \frac{1}{2} ( \coth (\lambda/2)z- i w, \coth(\lambda/2)w +i z)).$$
Observe that the series in \eqref{def-g-lambda} defining the kernel $ \widetilde{G}_\lambda(z,w) $ converges in $ \mathcal{S}_2 $ and we have
$$\sum_{\mu \in \Na^{2n}}   |  \zeta_\mu^\lambda(z,w)|^2  = K_{(z,w)}^\lambda(z,w) $$
where $ K_\zeta^\lambda(\zeta^\prime)$ is the reproducing kernel for $\Fs.$ Hence from the explicit expression for the reproducing kernel we get  the estimate
\begin{equation}\label{Glambda-esti}
| e^{\frac{1}{4} (\lambda/s_\lambda) (\bar{z}^2+\bar{w}^2)}| \| \widetilde{G}_\lambda(z,w)\|_{\mathcal{S}_2} \leq C_\lambda \, e^{-\frac{\lambda}{2} (u\cdot y- v\cdot x)}\, e^{\frac{1}{4}\lambda (\coth \lambda)(|z|^2+|w|^2)}.
\end{equation}
Thus $ \widetilde{G}_\lambda(z,w) $ is an entire function taking values in $ \mathcal{S}_2$ and it is evident that 
\begin{equation}\label{G-lambda-coeff} 
\pi^{-n/2} c_{n,\lambda}^{-1/2}\,e^{\frac{1}{4} (\lambda/s_\lambda) ({z}^2+{w}^2)} \, \tr \left(\widetilde{G}_\lambda(z,w) (S_\mu^\lambda)^\ast \right) = (-i)^{|\mu|}\, \zeta_\mu^\lambda(z,w).
\end{equation}
Under some additional condition on $ \mathcal{G}_\lambda(T)$ we can now prove the following integral representation for $ \widehat{T}.$

\begin{thm} Assume that $ T \in \mathcal{S}_2 $ is such that $ \mathcal{G}_\lambda(T)(\zeta) \in L^1(\C^{2n}, \sqrt{w_\lambda(\zeta)}\, d\zeta).$ Then
\begin{equation}\label{fourier} \widehat{T}= \pi^{-n/2} c_{n,\lambda}^{-1/2}\,\int_{\C^{2n}} e^{\frac{1}{4} (\lambda/s_\lambda) {\zeta}^2} \,  \mathcal{G}_\lambda(T)(\zeta) \widetilde{G}_\lambda(\zeta)\, w_\lambda(\zeta) \, d\zeta.
\end{equation}
In particular the above representation holds for any Schwartz class operator.
\end{thm} 
\begin{proof} From the estimate \eqref{Glambda-esti} and the hypothesis on $ \mathcal{G}_\lambda(T) $ it follows that the integral in \eqref{fourier} converges absolutely. In view of \eqref{G-lambda-coeff} and the fact that 
$$ ( \mathcal{G}_\lambda(T),  \zeta_\mu^\lambda)_{\Fs} =  (T, \mathcal{G}_\lambda^\ast \zeta_\mu^\lambda) = (T, S_\mu^\lambda)_{\mathcal{S}_2}$$
we see that $ (\widehat{T}, S_\mu^\lambda) = (-i)^{|\mu|}\, (T, S_\mu^\lambda)$ which proves the integral representation. For  any $ T \in \mathcal{S}_2 $ we have the estimate
$ |\mathcal{G}_\lambda(T)(\zeta)| \leq C\, w_\lambda(\zeta)^{-1/2} $ which can be improved if $ T = \pi_\lambda(f) $ with $ f \in \mathcal{S}(\R^{2n}).$ In fact it has been proved in \cite{RT}   that for a Schwartz class function one has the estimate
$$ | f \ast_\lambda p_t^\lambda(z,w)| \leq C_{m,\lambda} (1+|z|^2+|w|^2)^{-m} \,e^{-\frac{\lambda}{2} (u\cdot y- v\cdot x)}\, e^{\frac{\lambda}{2} (\coth 2t\lambda)(|y|^2+|v|^2)} $$
for any positive integer $ m.$
(See Theorem 4.4  in \cite{RT};  there is a misprint in the statement: $ \coth 4t $ should read $ \coth 2t $). With these estimates it is easy check that  $ \mathcal{G}_\lambda(T)(\zeta) = B_\lambda f(\zeta) \in L^1(\C^{2n}, \sqrt{w_\lambda(\zeta)}\, d\zeta).$ This completes the proof of the theorem.
\end{proof}

\begin{rem}  Given $ T \in \mathcal{S}_2 $ we can choose a sequence of Schwartz class operators $ T_k $ such that $ T_k \rightarrow T $ in $ \mathcal{S}_2.$  Then $ \widehat{T}_k $ defined by the integral \eqref{fourier} converges to $ \widehat{T} $ in $\mathcal{S}_2.$
\end{rem}

We will now show  that  for the Fourier transfom on $ \mathcal{S}_2 $ we have an analogue of Hardy's theorem.  Recall that in 1933, Hardy \cite{GHH} proved the following theorem in  the one dimensional case.

\begin{thm}   $ f(x) = c \,e^{-\frac{1}{2}|x|^2} $ are the only  functions   satisfying the following estimates:
$$ |f(x)| \leq C e^{-\frac{1}{2}|x|^2},\,\,\,\, |\widehat{f}(\xi)| \leq C e^{-\frac{1}{2}|\xi|^2}.$$
\end{thm}

Here is the exact analogue of Hardy's theorem for the  Fourier transform on operators. 

\begin{thm} \label{hardy} $ T = c \,e^{-\frac{1}{2}H(\lambda)}$  are the only operators in $ \mathcal{S}_2 $  satisfying  the bounds 
$$ T^\ast T \leq C e^{-H(\lambda)},\,\, \widehat{T}^\ast \widehat{T} \leq C e^{-H(\lambda)}.$$ 
\end{thm}
\begin{proof} According to a lemma of Douglas \cite{RD} the hypothesis on $ T $ and $ \widehat{T} $ implies that $ T = e^{-\frac{1}{2}H(\lambda)}\, M$  and 
$ \widehat{T} =  e^{-\frac{1}{2}H(\lambda)}\, M^\prime$ where $ M $ and $ M^\prime $ are bounded linear operators. It then follows that $ \varphi =\mathcal{G}_\lambda(T) $ and 
$\psi = \mathcal{G}_\lambda(\widehat{T}) $ both belong to the subspace $ \mathcal{A}^\lambda(\C^{2n}) $ studied in \cite{GT} and define bounded linear operators $ S_\varphi^\lambda $ and 
$ {S}_\psi^\lambda $ on $ \Fs.$  By definition, $ \psi = \mathcal{G}_\lambda(\widehat{T}) =  U \circ \mathcal{G}_\lambda(T)  = U\varphi.$  Then by the uncertainty principle proved in Theorem 1.4 in \cite{GT} we can conclude that $ \varphi $ is a constant and hence $ M = c I.$ This proves the theorem.
\end{proof}

\subsection{Multipliers and Fourier multipliers in the operator setting}In the proof of Theorem \ref{hardy} we have already encountered the operators $ S_\varphi^\lambda$ and $ \widetilde{S}_\psi^\lambda $ acting on the twisted Fock spaces $\Fs.$  It has been proved that the following relations hold:
$$ B_\lambda^\ast \circ S_\varphi^\lambda \circ B_\lambda =  T_M,\,\,\,  \widetilde{S}_\psi^\lambda = U^\ast \circ S_{U\psi}^\lambda \circ U  $$
where $ \varphi = G_\lambda(M), M \in B(L^2(\R^n)) $ and $ T_M $ is the Weyl multiplier operator defined by $ \pi_\lambda(T_Mf) =\pi_\lambda(f) M.$  We refer to \cite{GT} for these relations. Instead of conjugating with $ B_\lambda: L^2(\R^{2n}) \rightarrow \Fs $ we can also conjugate with $ \mathcal{G}_\lambda : \mathcal{S}_2 \rightarrow \Fs $ to get  operators acting  on $ \mathcal{S}_2.$ Consider first the operator 
$  \mathcal{G}_\lambda^\ast \circ S_\varphi^\lambda \circ  \mathcal{G}_\lambda  $ under the assumption that  $ \varphi = G_\lambda(M) $ so that the conjugated operator is bounded on $ \mathcal{S}_2.$
 Using the relation $ \mathcal{G}_\lambda^\ast (S_\varphi^\lambda F) = \mathcal{G}_\lambda^\ast(F)  M$ we can easily check the relation
$$ (\mathcal{G}_\lambda^\ast \circ  S_\varphi^\lambda  \circ \mathcal{G}_\lambda) T = T M.$$
On the other hand, since $ (U \circ \mathcal{G}_\lambda) T = \mathcal{G}_\lambda(\widehat{T}) = (\mathcal{G}_\lambda \circ \mathcal{F})T $ we also have the relation

$$ \mathcal{G}_\lambda^\ast \circ \widetilde{S}_\psi^\lambda  \circ \mathcal{G}_\lambda  = \mathcal{G}_\lambda^\ast \circ U^\ast \circ S_{U\psi}^\lambda \circ U \circ \mathcal{G}_\lambda = \mathcal{F}^\ast \circ \mathcal{G}_\lambda^\ast  \circ S_{U\psi}^\lambda \circ \mathcal{G}_\lambda \circ \mathcal{F}.$$
As a consequence of the above, if we let $ G_\lambda (M) = U \psi ,$ it follows that  the conjugated operator $ \mathcal{G}_\lambda^\ast \circ \widetilde{S}_\psi^\lambda  \circ \mathcal{G}_\lambda $ becomes a Fourier multiplier:
$$  (\mathcal{G}_\lambda^\ast \circ \widetilde{S}_\psi^\lambda  \circ \mathcal{G}_\lambda)T  = \mathcal{F}^\ast ( \widehat{T} M).$$ 
Thus the study of Fourier multipliers on $ \mathcal{S}_2 $ is equivalent to the study of the operators $ \widetilde{S}_\psi^\lambda $ on the twisted Fock spaces $ \Fs.$

\begin{rem}  It is an interesting open problem to find conditions on $ M $ so that the Fourier multiplier operator  $ T \rightarrow \mathcal{F}^\ast ( \widehat{T} M)$ extends to $ \mathcal{S}_p $ as a bounded operator. We plan to address this problem in a subsequent work.
\end{rem}

\subsection{Radial operators and their Fourier transforms}  When $ f \in L^2(\R^n) $ is radial we note that  $ (f, \Phi_\alpha) = 0 $ unless  each $ \alpha_j $ is even. This is due to the fact that the one dimensional Hermite functions $ h_k(x) $ are even (odd) when $ k $ is even (resp. odd). Thus all the nonzero Hermite coefficients in the expansion of $ f$ are of the form $ (f, \Phi_{2\alpha}).$ Moreover,
$$ (f, \Phi_{2\alpha}) = \int_{\R^n} f(y) \Phi_{2\alpha}(y) dy = \int_{\R^n} f(y) \left( \int_{SO(n)} \Phi_{2\alpha}(\sigma y) d\sigma \right) \, dy.$$
Since the inner integral is a radial eigenfunction of the Hermite operator with eigenvalue $ (4|\alpha|+n) $ it follows that it is a constant multiple of the Laguerre function
$$ \psi_k^{n/2-1}(y) = L_k^{n/2-1}(|y|^2)\, e^{-\frac{1}{2}|y|^2} $$
where $ k = |\alpha|.$ Therefore, the Hermite expansion of a radial function takes the form 
$$ f(x) = \sum_{k=0}^\infty R_k(f)\, \Big(\sum_{|\alpha|=k} c_\alpha \Phi_{2 \alpha}(x)\Big),\,\,\, R_k(f) = c_k \,\int_{\R^n} f(y) \psi_k^{n/2-1}(y)\, dy.$$
In analogy with this we say that an operator $ T \in \mathcal{S}_2 $ is radial if the coefficients $ (T, S_\mu^\lambda) = 0 $  unless $ \mu = 2 \nu $  in which case all the coefficients $ (T, S_{2\nu}^\lambda) $ with $ |\nu|=k $  fixed are proportional to each other. Thus  $ (T, S_{2\mu}^\lambda) = c_\mu^\lambda\, R_k^\lambda(T) $ where $ k = |\mu| $ and the expansion of $ T $ takes form
$$ T = \sum_{k=0}^\infty R_k^\lambda(T) \Big( \sum_{|\mu|=k} c_\mu^\lambda \, S_{2\mu}^\lambda \Big).$$
In the case of radial functions on $ \R^n $  it turns out that 
$\sum_{|\alpha|=k} c_\alpha \Phi_{2\alpha}(x) = c_{n,k}\,  \psi_k^{n/2-1}(x) $ and  the constants are related by 
$$ c_{n,k}^2 \,  \int_{\R^n} (\psi_k^{n/2-1}(x))^2\, dx =  \sum_{|\alpha|=k} |c_\alpha|^2  $$

We  are interested in exploring the operator 
$ P_{k,\lambda} =  \sum_{|\mu|=k} c_\mu^\lambda \, S_{2\mu}^\lambda .$ 
Considering the Hermite functions $ \Phi_\mu(x,u) $ on $ \R^{2n} $ we have the following formula 
$$\sum_{|\mu|=k} c_\mu \Psi_{2\mu}^\lambda(x,u) = c_{2n,k}\,  D_\lambda\psi_{k}^{n-1}(x,u) = c_{2n,k}\,  \sqrt{c(\lambda)}^n\, \psi_{k}^{n-1}(\sqrt{c(\lambda)}(x,u))  .$$
Applying the Weyl transform to both sides of the above equation we obtain
$$\sum_{|\mu|=k} c_\mu^\lambda \, S_{2\mu}^\lambda = c_{2n,k}\,  \sqrt{c(\lambda)}^n\, \int_{\R^{2n}} \pi_\lambda(x,u)\, \psi_{k}^{n-1}(\sqrt{c(\lambda)}(x,u))\, dx\, du  .$$
As the function $ \psi_k^{n-1}(x,u) $ is radial, its Weyl transform is a function of the Hermite operator, and hence has an expansion in terms of the spectral projections $ P_k(\lambda) $ associated to $ H(\lambda) $ on $ \R^n.$ This means that there is a function $ m_k $ on $ \mathbb N $ such that
$$\sum_{|\mu|=k} c_\mu^\lambda \, S_{2\mu}^\lambda =  \, \sum_{j=0}^\infty \, m_k((2j+n)\lambda)\, P_j(\lambda).$$
It may be possible to calculate $ m_k $ explicitly so that we have 
$$ P_{k,\lambda} = \sum_{|\mu|=k} c_{\mu}^\lambda \, S_{2\mu}^\lambda = m_k(H(\lambda)).$$

\begin{prop}\label{FT of radial op} An operator $ T \in \mathcal{S}_2 $ is radial if and only of $ T = m(H(\lambda)) $ for some function $ m.$ Moreover, $\widehat{T} $ is also radial whenever $ T $ is radial. 
\end{prop}
\begin{proof}  Assuming that  $ T $ is radial we have already shown  that $ T = \pi_\lambda(f) $ where
$$ f(x,u) = \sum_{k=0}^\infty c_{2n,k}\,R_k^\lambda(T)\,  \sqrt{c(\lambda)}^n\, \psi_{k}^{n-1}(\sqrt{c(\lambda)}(x,u))  .$$
It is known that a function $ g $ on $ \R^{2n} $ is radial, if and only $ \pi_\lambda(g) $ is a function of $ H(\lambda).$ Hence it follows that $ T = \pi_\lambda(f) = m(H(\lambda))$ for some function $ m.$
Conversely, if  $ T = m(H(\lambda)),$ so that  $ f = \pi_\lambda^\ast(T) $ is radial, we can expand $ g = D_\lambda^{-1}f $ in terms of $ \psi_k^{n-1} $ leading to 
$$ D_\lambda^{-1}f(x,u) = \sum_{k=0}^\infty  R_k(g)\, \psi_k^{n-1}(x,u) .$$
From the above it follows that the operator $ T $ has  the expansion 
$$ T =  \sum_{k=0}^\infty  c_{2n,k}^{-1}\, R_k(g)\, P_{k,\lambda} = \sum_{k=0}^\infty  \, \sum_{|\mu|=k} c_{2n,k}^{-1}\,R_k(g)\, c_\mu^\lambda \, S_{2\mu}^\lambda.$$
From the above it is now clear that  for any $ \mu, \nu $  with $ |\mu|=|\nu|=k,$ we have $ (T, S_{2\mu}^\lambda) = R_k(g) c_\mu^\lambda $ and  $ (T, S_{2\nu}^\lambda) = R_k(g) c_\nu^\lambda $ so that
$ T $ is radial by our definition.  We have shown that every radial $ T $ has the expansion
$$ T = \sum_{k=0}^\infty R_k^\lambda(T)\, \Big( \sum_{|\mu|=k} c_\mu^\lambda \,  S_{2\mu}^\lambda \Big).$$ 
Since $ S_\mu^\lambda $ are eigenvectors of the Fourier transform with eigenvalues $ (-i)^{|\mu|} $ it is clear that $ \widehat{T} $ is radial whenever $ T $ is radial.
\end{proof}

\begin{rem} Though the Hermite functions $ \Phi_\alpha $ on $ \R^n $ are not radial, there are constants $ c_\alpha $ such that the particular combination $ \sum_{|\alpha|=k} c_\alpha\, \Phi_{2\alpha} $ is radial.
Similarly, the operators $ S_\mu^\lambda $ are not radial for $ \mu \neq 0$ but a suitable combination $ \sum_{|\mu|=k} c_\mu\, S_{2\mu}^\lambda $ is radial.
\end{rem}

The Laguerre functions $\psi_k^{n-1}(x,u) $ form an orthogonal system in $ L^2(\R^{2n}) $ and hence the operators $ P_{k,\lambda} = \pi_\lambda(D_\lambda \psi_k^{n-1}) $ form an orthogonal family of operators in $ \mathcal{S}_2.$ We call them operators of Laguerre type for obvious reasons. Any radial operator $ T $ has an expansion in terms of $ P_{k,\lambda}.$

\begin{rem} Recall the Hecke-Bochner identity for the Fourier transform $ \mathcal{F}_n$ on $ \R^n $ which can be stated as follows. Let  $ f $ be a radial function on $ \R^n $ which can  naturally be treated as radial function on any $ \R^d.$ Let $ P $ be a solid harmonic of degree $ m.$ Then 
$$ \mathcal{F}_n(P f) =  P \, \mathcal{F}_{n+2m}(f).$$
From the work of Geller  \cite{DG} we have an analogue of Hecke-Bochner identity for the Weyl transform acting on $ L^2(\R^{2n}).$  The operator analogues of spherical harmonics turned out to be the Weyl correspondence  of bigraded spherical harmonics, see \cite{DG} and \cite{ST-uncertainty}.  Using  Geller's formula, we can  establish a Hecke-Bochner identity for the Fourier transform on $ \mathcal{S}_2.$ The details of this formula and its connection with certain class one representations of the unitary group $ U(n) $ will be taken up in a forthcoming work. 
\end{rem}



\section*{Acknowledgements}

This work  began in May 2024 when the second author (ST) was visiting IISER, Bhopal. It took shape  slowly  as ST was visiting  Ghent  University and I.I.T, Mumbai and was completed recently in I.I.Sc, Bangalore. ST wishes to thank these institutes for their warm hospitality and  facilities provided. He also wishes to thank Hasan Ali Biswas for checking the manuscript for correctness of constants.


\end{document}